\documentclass[a4paper, english, 10pt]{smfart} 

\usepackage{setspace}
\usepackage[british]{babel}

\usepackage[OT2,T1]{fontenc}
\usepackage{amssymb} 
\usepackage{mathrsfs} 

\usepackage{mathdots} 

\usepackage{stmaryrd}
\usepackage{amsthm}

\usepackage{amsmath}
\usepackage[latin1]{inputenc}

\allowdisplaybreaks[1]

\usepackage{graphicx}
\usepackage{manfnt} 

\usepackage{smfthm}

\usepackage[pagebackref=true, colorlinks=true, linkcolor=black, citecolor=black, urlcolor=black]{hyperref}

\usepackage[msc-links]{amsrefs}

\usepackage[margin=3cm]{geometry}

\usepackage[all]{xy}

\usepackage{math tools} 

\usepackage{enumitem}

\DeclareSymbolFont{cyrletters}{OT2}{wncyr}{m}{n}
\DeclareMathSymbol{\Sha}{\mathalpha}{cyrletters}{"58}

\newtheorem*{coro*}{Corollary}
\newtheorem*{conj*}{Conjecture}
\newtheorem*{lemm*}{Lemma}

\providecommand{\smalltwomat}[4]{\left(\begin{smallmatrix}#1&#2\\#3&#4\end{smallmatrix}\right)}

\newcommand{\beq}{\begin{equation}\begin{aligned}}
\newcommand{\eeq}{\end{aligned}\end{equation}}

\newcommand{\beqq}{\begin{equation*}\begin{aligned}}
\newcommand{\eeqq}{\end{aligned}\end{equation*}}

\theoremstyle{definition}

\theoremstyle{remark}
\newtheorem{remark*}{Remark}

\NumberTheoremsIn{subsection}

\numberwithin{equation}{subsection}

\newcommand{\bmu}{\boldsymbol{\mu}} 

\newcommand{\lb}[1]{\label{#1}}
\newcommand{\ts}{\times}

\newcommand{\one}{\mathbf{1}}
\newcommand{\Q}{\mathbf{Q}}

\newcommand{\GL}{\mathrm{GL}}
\newcommand{\G}{\mathrm{G}}

\newcommand{\X}{\mathscr{X}}

\newcommand{\cH}{\mathscr{H}}

\newcommand{\la}{\langle}
\newcommand{\ra}{\rangle}

\newcommand{\cW}{\mathscr{W}}

\newcommand{\Z}{\mathbf{Z}}

\newcommand{\frakp}{\mathfrak{p}}

\newcommand{\into}{\hookrightarrow}

\newcommand{{\calG}}{\mathscr{G}}

\newcommand{\C}{\mathbf{C}}

\newcommand{\N}{\mathbf{N}}
\newcommand{\OO}{\mathscr{O}}

\newcommand{\bks}{\backslash}
\newcommand{\baar}{\overline}

\newcommand{\eps}{\varepsilon}

\newcommand{\vpi}{\varpi}

\newcommand{\wtil}{\widetilde}

\newcommand{\W}{\mathscr{W}}

\newcommand{\vol}{\mathrm{vol}}

\newcommand{\lm}{\lambda}

\newcommand{\sg}{\sigma}

\newcommand{\Gal}{\mathrm{Gal}}

\newcommand{\Ker}{\mathrm{Ker}\,}

\newcommand{\Hom}{\mathrm{Hom}\,}
\newcommand{\End}{\mathrm{End}\,}

\newcommand{\ot}{\otimes}

\newcommand{\llb}{\llbracket}
\newcommand{\rrb}{\rrbracket}

\newcommand{\Spec}{\mathrm{Spec}\,}

\newsavebox\tempbox
\let\svwidetilde\widetilde
\renewcommand\widetilde[1]{\sbox\tempbox{$#1$}\svwidetilde{\usebox{\tempbox}}}

   \def\XXint#1#2#3{{\setbox0=\hbox{$#1{#2#3}{\int}$}
        \vcentre{\hbox{$#2#3$}}\kern-.5\wd0}}

\title[Local Langlands in analytic families]{Local Langlands correspondence, local factors,\\and zeta integrals  in analytic families}
\author{Daniel Disegni} 
\address{Department of Mathematics, Ben Gurion University of the Negev, Be'er Sheva 84105, Israel}
\email{daniel.disegni@gmail.com}

\begin{document}

\begin{abstract} We study the variation of the local Langlands correspondence for $\GL_{n}$ in characteristic-zero families. We establish an existence and uniqueness theorem for a correspondence in families, as well as a recognition theorem for when a given pair of  Galois-  and reductive-group- representations can be identified as local Langlands correspondents. The results, which can be used to study local-global compatibility questions along eigenvarieties,   are largely analogous to those  of Emerton, Helm,  and Moss on the local Langlands correspondence over local rings of mixed characteristic. We apply the theory to the interpolation of local zeta integrals and of $L$- and $\eps$-factors.  
\end{abstract} 

\thanks{2010 Mathematics Subject Classification:  11F33; 11F85; 11F80}

\maketitle
\tableofcontents

\section{Introduction}
The aim of this work is to study the variation in characteristic-zero families of the local Langlands correspondence for $\GL_{n}$,  with applications to the variation  of local $L$-functions, $\eps$- and $\gamma$-factors,  and zeta integrals.  

Our results on the local Langlands correspondence are analogous to those of  Emerton, Helm, and Moss, who in an important series of papers \cite{EH, helm-bern, helm-whitt, moss-gamma, DL, moss-interpol, helm-curtis, hm-conv}  
studied the variation of such objects in families over complete local rings of mixed characteristics; we refer to those works collectively as [Mixed]. The proofs use in a crucial way  some of their ideas (but certainly  not all: our context  is less delicate).  

We hope nevertheless that the present work will be a useful contribution  to the literature: families over global  characteristic-zero bases arise naturally in the context of eigenvarieties, and there are also many examples of   Galois-modules over local $\Z_{p}$-algebras $R$ which are  flat as $R[1/p]$-modules but not as $R$-modules;\footnote{A first example occurs when $R=\mathbf{T}_{\frak{m}}$ is a non-Gorenstein local Hecke  algebra as in \cite{kilford} (or a corresponding big Hecke algebra), and we consider the first homology  (possibly completed) of a modular curve localised at $\frak{m}$.}  such families lie outside the framework of [Mixed]. 
 Moreover  we pay special attention to questions of rationality in the coefficients.
 Finally, readers interested in [Mixed] may still find this paper a useful point of entry before diving into the subtleties of that context:  while we do refer to  [Mixed] for the proofs of several lemmas, the whole theory here  is presented  in an essentially self-contained fashion. 

In \S\ref{main} we  describe  our main results  and the organisation of the  paper. In \S\ref{appl} we sketch  the typical applications  that we have in mind.
 
\subsection{The results} \label{main}   Let $F$ be a non-archimedean local field. Let  $W_{F}$ (resp. $W_{F}'$) the Weil (resp. Weil--Deligne) group of $F$, and let $G=G_{n}:=\GL_{n}(F)$.
Representations of $W_{F}$ and $W_{F}'$  will be tacitly understood to be (Frobenius-) semisimple; 
representations of $G$ will be  understood to be  smooth. Let $K$ be a field of characteristic zero and ${\rm Noeth}_{K}$ be the category of reduced Noetherian schemes over $K$. 

In \S\ref{sec: pi gen} we recall basic facts from the theory of Bernstein--Zelevinsky and its classification of representations of $G_{n}$ in terms of  supercuspidal supports and multisegments of such. We then define the \emph{generic} local Langlands correspondence
$$r'\mapsto \pi_{\rm gen} (r') $$
 after Breuil--Schneider: it is a map  from Frobenius-semisimple $n$-dimensional Weil--Deligne representations over $K$ (up to isomorphism) to  indecomposable finite-length  representation of $G_{n}$ over $K$ which are generic (i.e. admit a unique Whittaker functional);  it agrees with the Tate-normalised Langlands correspondence $\pi(\cdot)$ for those $r'$ such that $\pi(r')$ is generic, and it  is compatible with automorphisms of the coefficient field. In \S\ref{sec: LL Ber} we define the rational Bernstein `variety'  ${\frak  X}_{n}/$: it is a scheme locally of finite type over $\Q$, whose base-change to $\C$ has the usual Bernstein centre of $G_{n}$ as ring of functions. The points of ${\frak X}_{n}$ are naturally in bijection with Galois orbits of supercuspidal supports for $G_{n}$. We show in Theorem \ref{coarse} that its base-change ${\frak X}_{n,K}$ coarsely pro-represents  the functor on ${\rm Noeth}_{K}$ which sends $X$ to the set of isomorphism classes  of rank-$n$  representations of $W_{F}$ over $X$. This result (the \emph{semisimple} local Langlands correspondence in families) is  essentially due  to Helm \cite{helm-curtis}. In \S\ref{sec: ext} we define an extension ${\frak X}_{n}'\to {\frak X}_{n}$ of the Bernstein variety, whose points are in bijection with Galois orbits of  multisegments of supercuspidal supports. We then show that $\frak X_{n}'$ coarsely pro-represents the functor on ${\rm Noeth}_{K}$  of rank-$n$  families of Weil--Deligne representations with locally constant monodromy. 
 
In \S\ref{sec: LLF} we define, for any object $X$  of ${\rm Noeth}_{K}$, the key notion of a \emph{co-Whittaker} $\OO_{X}[G_{n}]$-module, after Helm: when $X=\Spec \C$ this singles out a class of generic $G_{n}$-representations containing the image of $\pi_{\rm gen}$. Given an $n$-dimensional Weil--Deligne representation $r'$ over $X$, we show that there exists a torsion-free co-Whittaker module $\pi(r') $ whose special fibre over  each generic point  (in the sense of algebraic geometry) $x\in X$ satisfies  
\beq\label{sp intro} \pi(r')_{|x}\cong  \pi_{\rm gen}(r'_{|x}).
\eeq
Moreover such $\pi(r')$ is unique up to twisting by a line bundle with trivial $G_{n}$-action and it satisfies \eqref{sp intro} at all points $x$ lying in a unique irreducible component of $X$ (in fact a  weaker condition suffices): see Theorem \ref{theo-LLF}. The construction of $\pi(r')$ is based on  the special  but almost universal case  $X={\frak X}_{n}$.

We then prove a recognition theorem (Theorem \ref{recognise}): given a rank-$n$ Weil--Deligne representation $r'$,  a finitely-generated $G_{n}$-representation $V$ over $X$, and a dense subset $\Sigma\subset X$ such that $V_{|x}=\pi(r'_{|x})=\pi_{\rm gen}(r'_{|x})$ for all $x\in \Sigma$, we show that there exists an open $U\supset X$ containing $\Sigma$ such that $V_{|U}\cong \pi(r'_{|U})$, up to tensoring with a line bundle. (The open $U$ is the locus where $V$ is co-Whittaker.) Our results hold more generally for representations of the group $\prod_{v\in S}\GL_{n}(F_{0,v})$, where the product ranges over a finite set of non-archimedean completions of a number field $F_{0}$.

Finally, in \S\ref{sec: Lz} we show that when $r_{1}'$, $r_{2}'$  are Weil--Deligne representations with locally constant monodromy over $X$ of ranks $n_{1}\geq n_{2}$ respectively, and $\pi_{i}=\pi(r_{i}')$, the 
  local Rankin--Selberg zeta integrals for  $\pi_{1}\times \pi_{2}$, divided by the corresponding local $L$-value, interpolate to a regular function on $X$. The key ingredient is  a result of Cogdell and Piatetski-Shapiro \cite{cps}, which in our language is seen as the case where $X$ is a cover of a connected component of ${\frak X}_{n, \C}'$.
We use these results to provide an interpolation of local $\gamma$- and $\eps$- factors in families.
  
  The same principles apply to many other  integrals of use in the representation theory of $G_{n}$: as a sample we show that the standard invariant inner product on unitarisable generic representations of $G_{n}$ can also  be interpolated in families parametrised by Weil--Deligne representations with locally constant monodromy.

\subsection{Intended applications}\label{appl}
 The typical situation where our results apply  is roughly  the following. Let $p$ be a rational prime. Let $\G$ be a reductive group over $\Q $ such that $\G({\Q_{\ell}})\cong \GL_{n}(\Q_{\ell})=:G_{n}$ for  a prime $\ell\neq p$.\footnote{More generally, the same ideas apply to the following case. Let   $F_{0}\subset F_{1}$ be an extension of number fields, let $S_{0}$ be a finite set of finite places of $F_{0}$ disjoint from those above $p$, and let $S_{1}$ be the set of places of $F_{1}$ above $F_{0}$. Then we may consider a reductive group  $\G$ over $F_{0}$ such that $\prod_{v\in S_{0}}\G_{F_{0,v}}\cong\prod_{w\in S_{1}}\GL_{n/ F_{1,w}} $.}
 Let   $X=\Spec R$ where  $R$ is either the coordinate ring of an affinoid open of a $p$-adic eigenvariety for the group $\G$, or $R$ is a localisation of $R^{\circ}[1/p] $ where $R^{\circ}$ is the base of a $p$-adic Hida family for $\G$. Let $\Sigma\subset X$ be the set of classical points and assume it is dense (otherwise we replace $X$  by the Zariski closure of $\Sigma$). One can often construct a family $\rho$ of $n$-dimensional representations of the Galois group of $\Q_{\ell}$ over $X$ and, by global methods,  a finitely generated  torsion-free $\OO_{X}[G_{n}]$-module $\Pi$.
 
\subsubsection{Local-global compatibility along eigenvarieties} In the above situation,  one often knows that (i) $\Pi_{|x}$ is generic (e.g. because it is  the local component of a global cuspidal automorphic representation), and (ii)
 points $x\in \Sigma $ satisfy local-global compatibility, that is that the Weil--Deligne representation $r'_{|x}:= {\rm WD}(\rho_{|x})^{\text{Fr-ss}} $ is associated with $\Pi_{|x}$ by the local Langlands correspondence $\pi$. As recalled in Lemma \ref{l-mon}, $\rho $ itself admits a Weil--Deligne representation  $r'={\rm WD}(\rho)$. 
 Then the recognition theorem \ref{recognise} allows to identify, up to twisting with a line bundle,  the $\OO_{X}[G_{n}]$-modules  $\Pi$ and $\pi(r')$ over an open subset $X'\subset X$ containing $\Sigma$; in other words, the family $\Pi_{|X'}$ satisfies local-global compatibility. By the uniqueness property this result can be glued to cover cases where $X$ is a more general rigid space. 
 
\subsubsection{Interpolation of local integrals} Suppose moreover that one wishes to interpolate some local zeta integrals $Z_{x}\colon \Pi_{|x}\to \Q_{p}(x)$, say defined in a Whittaker model of $\Pi_{|x}$. Then one can do so over $X'$ by the local-global compatibility and the results of \S\ref{sec: Lz}.  A typical situation where the interpolation of local integrals is useful is the construction of $p$-adic $L$-functions (e.g. as in \cite{pyzz}). There, one  starts with an integral representation of the form 
$$Z(\phi_{|x}) = L(\Pi_{x})\prod_{v}Z_{v}(\phi_{v|x}),$$ where now $\Pi$ is a family of (usually, $p$-adic Jacquet modules of) automorphic representations over $X$, $Z(\phi_{|x})$ is a global integral of $\phi\in \Pi$, and $Z_{v}(\phi_{v|x})$ are local integrals on $\Pi_{v|x}$ (here $\phi=\otimes_{v}\phi_{v}$ is a decomposition). After interpolating $x\mapsto Z(\phi_{|x})$ or its modification through a Hecke operator at $p$, one may either (i) choose $\phi$ carefully at all $v\nmid p$, so that $Z_{v}(\phi_{v|x})$ is explicitly computable to be a constant or some simple function on $X$, or (ii) interpolate the $Z_{v}(\phi_{v|x})$ for $v\nmid p$ along the lines of \S\ref{sec: Lz}, and define 
$$L_{p}(\Pi):= {Z(\phi)\over \prod_{v\nmid p} Z_{v}(\phi_{v})}$$ 
for any $\phi\in \Pi$.\footnote{As $\Pi$ will usually interpolate the tensor product of $\Pi_{|x}^{p}$ and of the finite-dimensional Jacquet modules (or ordinary parts) of $\Pi_{|x}^{p}$  for automorphic representations $\Pi_{|x}$, the  vector   $\phi_{v}$  for $v\vert p$ is restricted.  The product of suitably modified $p$-adic local integrals
 $\prod_{v\vert p } Z_{v}(\phi_{v|x})$ will be the interpolation factor at $x$.}
The approach (ii) is more flexible and less laborious, as it does not require ad hoc choices and explicit calculations. When $L_{p}(\Pi)$ is expected to enjoy some integrality or holomorphicity properties,   however,   additional work may be required to establish them.

\subsubsection{An example} For a  specific example of all of the above, see \cite{dd-univ}.
 
 \subsubsection{Related works} The  works \cite{Pau11, JNS},  as well the earlier \cite[Appendix B]{urban-adj},  implicitly deal with the correspondence defined here, studying  local-global compatibility for specific eigenvarieties, or Hida families,  by different methods. 
 
\subsection{Acknowledgements} The debt this work owes to [Mixed] is clear.  I would like to thank  David Helm for  guidance at an early stage and Olivier Fouquet,  Gil Moss, Robert Pollack, Eitan Sayag and  Eric Urban for useful conversations.  Finally, I am grateful to the referee for a careful reading.

During the preparation of the first version of this paper, the author was affiliated with the Laboratoire de Math\'{e}matiques d'Orsay and supported by a public grant  of the Fondation Math\'ematique Jacques Hadamard.

\section{Generic local Langlands correspondence }\label{sec: pi gen}
In this section, we set up the framework and recall various versions of the local Langlands correspondence, most importantly (a variant of) one due to Breuil--Schneider, which will later be shown to exhibit a reasonable behaviour in families.
\subsection{Weil--Deligne representations} 
Let $F$ be  a non-archimedean local field with residue field of cardinality $q$, and let $G_F$ (resp. $W_F$, $I_F$, $W_F'$)  be its absolute Galois group (resp. Weil group, inertia group, Weil--Deligne group). We denote by $\phi
\in W_F$ a geometric Frobenius and, if $r$ is a representation of $W_F$, by $r(i)$ the representation $r\otimes |\ |^i$, where $|\ |\colon W_F/I_F\to \Q^\times$ is the homomorphism characterised by $|\phi|=q^{-1}$.  Representations of these groups will always be understood to be continuous.
 
Recall that a representation of $W_F'$, called a Weil--Deligne representation, with values in a $\Q$-algebra $B$ 
`is'  a pair $(r, N)$ where $r\colon W_F\to B^\times $  is a  group homomorphism, trivial on some open subgroup of $W_{F}$, and $N\in B$ is a nilpotent element  satisfying $N r(\phi) =q r(\phi)N$. It is called Frobenius-semisimple if $r$ is semisimple.

 The classification of  Weil--Deligne representations with values in  $B=M_n(C)$, for an algebraically closed field $C$ of characteristic zero, is well known. The Frobenius-semisimple  indecomposable objects are of the form 
\beq\lb{speh}
{\rm Sp}({r,m})=r\oplus\ldots\oplus r(m-1)  
\eeq for some irreducible continuous $W_{F}$-representation $r$ and an integer $m\geq 1$, with $N$ mapping $r(i)$ isomorphically onto $r(i+1)$ for all                                                                                                                                                                                                                                                                                                                                                                                                                                                                                                                                                                                                                 $i<m-1$. 
 Any Frobenius-semisimple Weil--Deligne representation is a direct sum $r' = \bigoplus {\rm Sp}({r_{i},m_{i}})$ of indecomposable ones.  
 
 If $X $ is a scheme and  $V$ is a locally  free  $\OO_{X}$-module, a Weil--Deligne representation {on $V$} is one valued in $B={\rm End}_{\OO(X)}(V(X))$.  If $X=\Spec K$ is the spectrum of a field, a Weil--Deligne representation on $V$ induces a monodromy filtration on $V$  \cite[(1.5.5)]{ill94}.
If   $K$ is of characteristic zero,     a Frobenius-semisimple Weil--Deligne representation   is said to be  \emph{pure} 
 of weight $w\in \Z$ if  all  eigenvalues $\lambda$ of $\phi$ on ${\rm Gr}_{i}(V_{\bullet})$ are $q$-Weil numbers of weight $w + i$ (that is, $q^{n}\lambda$ is an algebraic integer for some $n$, and for all $\iota \colon \Q(\lm)\into \C$ we have $|\iota(\lambda)|=q^{(w+i)/2}$).

\subsubsection{Grothendieck's $\ell$-adic monodromy theorem in $p$-adic families} We recall  how to associate a Weil--Deligne representation to certain  families of representations of $\Gal(\baar{F}/F)$.  This result is not used in the rest of the paper but it is useful for applying the theory developed here. 
Let $p$ be a prime different from the residue characteristic of $F$. We say that $A$ is a \emph{$p$-adic ring} if either (i) $A$ is an affinoid algebra over $\Q_{p}$,  or (ii) there is a complete Noetherian local domain $A^{\circ}$ of characteristic zero and  residue characteristic $p$, with fraction field $\mathscr{K}$, such that $A$ is a subring of $\mathscr{K}$ containing $A^{\circ}[1/p]$.  If $A$ is a $p$-adic ring, we say that a representation $\rho$ of $\Gal(\baar{F}/F)$ on a finite locally free $A$-module $M $ is \emph{continuous } if, respectively, (i) $\rho $ is continuous for the topology induced from the topology of $M$, or (ii) there is a Galois-stable finite $A^{\circ}$-submodule $M^{\circ}$ such that $M^{\circ}\otimes_{A^{\circ}}A=M$ and the restriction of $\rho$ to $M^{\circ}$ is continuous for the topology inherited from $A^{\circ}$.
\begin{lemm}\label{l-mon}
Let $A$ be a $p$-adic ring, let $M$ be a locally free $A$-module,  and let $\rho\colon \Gal(\baar{F}/F)\to {\rm Aut}_{A}(M)$ be a representation. Fix  a nontrivial homomorphism   $t_{p}\colon  \Gal(\baar{F}/F) \to \Q_{p}$.
 If $\rho $ is continuous in the sense of the previous paragraph, it uniquely determines a Weil--Deligne representation 
$$r'=(r, N)=: {\rm WD}(\rho)$$
on $M$
 such that $\rho $ coincides with $g\mapsto \exp(t_{p}(g)N)$ on some open subgroup of $I_{F}$. 
\end{lemm}
\begin{proof} See   {\cite[Lemma 7.8.14]{BCh}} for case (i) and \cite[Proposition 4.1.6]{EH} for case (ii).
\end{proof}

\subsection{Theory of Bernstein--Zelevinsky}\lb{21}
We let $G=G_{n}=\GL_{n}(F)$, and denote by $K$ a field of characteristic zero.

Let $P\subset G$ be a parabolic subgroup with Levi $M=\prod_{i=1}^{t}G_{d_{i}}$. We recall the Langlands classification (due to Bernstein--Zelevinsky) and correspondence (proved by Harris--Taylor and Henniart), according to two different normalisations, the \emph{unitary}  and the \emph{rational}; we will  use a subscript $?\in\{{\rm u}, \emptyset\} $ (respectively) to distinguish them; our discussion is based on \cite[\S 3]{clozel}. There are correspondingly two notions of induction ${}_{?}{\rm I}_{P}^{G}$ from a parabolic subgroup $P\subset G$ with Levi $M\cong \prod G_{d_{i}}$:  one,  ${}_{\rm u}{\rm I}_{P}^{G} $, is  the unitarily normalised induction of representations $\sigma=\bigotimes\sigma_{i}$, defined if $K$ contains a (fixed) square root of $q$; the other, defined by  
$${\rm I}_{P}^{G}(\sigma_{1}\otimes\ldots \otimes \sigma_{t}):= {}_{\rm u}{\rm I}_{P}^{G}(\sigma_{1}|\ |^{1-d_{1}\over 2}\otimes \ldots\otimes \sigma_{t}|\ |^{1-d_{t}\over 2}) |\ |^{n-1\over 2},$$
over $K(\sqrt{q})$, descends to a functor ${\rm I}_{P}^{G}$ (rational induction)  on representations over $K$, see \cite[\S 3.4.3]{clozel}.

 A \emph{supercuspidal support} for $G$  over   $K$ is an equivalence class of pairs $(M, \sigma)$ where $M\subset G$ is a Levi subgroup and $\sigma $ is a supercuspidal representation of $M$ over $K$, up to the relation of $G$-conjugation of both.  Equivalently, it is a multiset $\{\sigma_{1}, \ldots , \sigma_{m}\}$ where each $\sigma_{i}$ is a supercuspidal representation of $G_{d_{i}}$ for some partition $n=\sum d_{i}$.
 
 Let $\pi$ be a smooth admissible irreducible  representation of $G$ over $K$. For each choice (rational or, provided $K=K(\sqrt{q})$, unitary) of normalisation, the set of supercuspidal representations $\sigma=\sigma_{1}\otimes \ldots\otimes \sigma_{r}$ of Levi subgroups $M=\prod\GL_{d_{i}}$ such that   $\pi$ is a Jordan--H\"older constituent of the parabolic induction  ${}_{\rm u}{\rm Ind}_{P}^{G} \sigma$ (resp. ${\rm I}_{P}^{G}\sigma$), for some parabolic $P\subset G$ with Levi $M$, is an orbit under the action of $G$ conjugacy. That is, it is a supercuspidal support, called \emph{the unitary (resp. rational) supercuspidal support} of $\pi$. 
 If the smooth representation $\pi$ is only assumed to be of finite length , we say that it has the  supercuspidal support $\sg$ if each of the Jordan--H\"older constituents has supercuspidal support $\sg$. 

\subsubsection{Theory of Bernstein--Zelevinsky} A \emph{segment} for $G_{d}$ over $K$ is a  set of representations of $G_{d}$ of the form $\Delta=\Delta({\sigma, m})=\{\sigma, \ldots, \sigma(m-1)\}$, with $\sigma $ a supercuspidal representation of $G_{d}$ over $K$ (so $n=md$). The generalised Steinberg representation $$\pi_{?}(\Delta):={}_{?}{\rm St}_{\Delta}$$
is the unique irreducible generic\footnote{That is, admitting a unique Whittaker model: see more precisely the definition of being \emph{of Whittaker type} (one generalisation of the  notion of being \emph{generic} applicable beyond irreducible representations over a field) in Definition \ref{coW def} below.} representation with unitary (resp. rational) supercuspidal support $\Delta$. It is the unique irreducible quotient of ${}_{?}{\rm I}_{P}^{G}(\sigma\otimes \ldots\otimes \sigma(m-1))$ and the unique irreducible subrepresentation of ${}_{?}{\rm I}_{P}^{G}(\sigma(m-1)\otimes \ldots\otimes \sigma)$.  The essentially square-integrable irreducible representations of $G$ are exactly those of the form $\pi_{?}(\Delta)$ for some segment $\Delta$.

Let  ${\bf s}=\{\Delta_{1}, \ldots, \Delta_{t}\}$ be a \emph{multisegment} (that is a multiset of segments), and let $\sigma$ be its supercuspidal support, that is the  multiset of supercuspidal representations $\sigma_{i}(j)$ of $G_{d_{i}}$ occurring in the $\Delta_{i}$. We say that $\Delta_{i}=\Delta(\sigma_{i}, m_{i})$ \emph{precedes} $\Delta_{j}=\Delta(\sigma_{j}, m_{j})$ if none of them contains the other and $\sigma_{j}=\sigma_{i}(t+1) $ for some $0\leq t\leq m_{i}-1$.
Suppose that the  $\Delta_{i}$ in ${\bf s}$ are ordered so that for $i<j$, $\Delta_{i}$ does not precede $\Delta_{j}$. 
Then 
$$\pi'_{?, {\rm gen}}({\bf s}):= {}_{?}{\rm I}_{P}^{G}(\pi_{?}(\Delta_{1}), \ldots, \pi_{?}(\Delta_{r})),$$ the  induction with respect to any  suitable parabolic $P$, is independent of the particular  choice of ordering  (among those satisfying the condition mentioned) and of $P$ up to isomorphism;  it has a unique irreducible  subrepresentation, which is the unique generic irreducible representation with supercuspidal support $\sigma$,  and a unique irreducible quotient denoted $\pi_{?}({\bf s})$. 

Suppose that the $\Delta_{i}$ in ${\bf s}$ are ordered so that for all $i<j$, $\Delta_{j}$ does not precede $\Delta_{i}$.  
Then the representation
\beq\lb{rdnp}
\pi_{{?}, {\rm gen}}({\bf s}):={}_{?}{\rm I}_{P}^{G}(\pi_{?}(\Delta_{1}), \ldots, \pi_{?}(\Delta_{r}))\eeq
  is independent of the particular  choice of ordering   (among those satisfying the condition mentioned) up to isomorphism;  it has a unique irreducible quotient, which is the generic representation with  supercuspidal support $\sigma$, and a unique irreducible subrepresentation, which is the representation $\pi_{\rm ?}({\bf s})$ of the previous paragraph.

\subsection{Langlands correspondences}
We continue with the notation of \S \ref{21}. The (unitary or rational) local Langlands correspondences $\pi_{\rm u}$, $\pi$ (denoted collectively as $\pi_{?}$) are   bijections between equivalence classes of  Frobenius-semisimple $n$-dimensional Weil--Deligne  representations over $\C$ and  of  smooth admissible irreducible representations of $G_{n}$ over $\C$. (The rational normalisation is also known as \emph{geometric} or \emph{Tate} normalisation.) They are related by 
$$\pi(r')=\pi_{\rm u}(r')|\ |^{n-1\over 2},$$
and if $r'^{*}$ denotes the dual representation to $r'$ then 
\beq\label{duality}
\pi_{\rm u}(r'^{*})\cong \pi_{\rm u}(r')^{\vee}\eeq (smooth dual).
 Each of $\pi_{\rm u}$, $\pi$  restricts to a bijection between irreducible Weil--Deligne representations (thus of the form $r'=(r, 0)$ where $r$ is an irreducible $W_{F}$-representation) and supercuspidal $G_{n}$-representations, and the whole correspondence for $G_{n}$ is deduced from such restricted version for all $m\leq n$ as follows.   Let $$r' = \bigoplus {\rm Sp}({r_{i},m_{i}}).$$
  Let $\sigma_{{\rm ?}, i}:=\pi_{\rm ?}((r_{i},0))$, and consider the 
multisegments 
$${\bf s}_{{\rm ?}, r'}:=\{\Delta({\sigma_{{\rm ?}, i}, m_{i}})\}$$
Then    $$\pi_{\rm ?}(r')=\pi({\bf s}_{ ?, r'}).$$ 
The rational  correspondence $\pi$ is compatible with automorphisms of $\C$.  

The correspondence $r'\mapsto \pi_{\rm gen}'({\bf s}_{r'})$ is the one defined by Breuil--Schneider in \cite[pp. 162--164]{BS} and recalled in \cite[\S 4.2]{EH}. We will be more concerned with the correspondence 
\begin{align}\label{pigen}
\pi_{\rm gen}\colon r'\mapsto \pi_{\rm gen}({\bf s}_{r'}),\end{align}
which we call the \emph{generic} (rational) local Langlands correspondence, related to $\pi_{\rm gen}'$ by 
\beq\lb{dualgen}
\pi_{\rm gen}(r')^{\vee}=\pi_{\rm gen}'({r'}^{*}).
\eeq
 By \cite[Lemma 4.2]{BS} and the argument at the end of \cite[\S 4.2]{EH}, it descends to a map still denoted by $\pi_{\rm gen}$ from Weil--Deligne representations over $K$ to $G$-representations defined over $K$, for any\footnote{In \emph{locc. citt.}, the coefficient field is supposed to contain $\Q_{p}$, but this is not necessary for any of their statements or proofs, until one wants to pass from Weil--Deligne representations to Galois representations.}
 characteristic-zero field $K$.  (The same is true of the rational Langlands correspondence $r'\mapsto \pi(r')$; the corresponding statements are true for the unitary normalisation provided $K$ contains $\sqrt{q}$.)

Finally, the correspondence $\pi_{?}$ induces a bijection  (\emph{semisimple} Langlands correspondence)
\beq \label{piss}
\pi_{?, {\rm ss}}\colon r\mapsto  (M, \sigma) 
\eeq
between semisimple $n$-dimensional representations of  $W_{F}$ over $\C$ and supercuspidal supports for $G_{n}$ over $C$. The correspondence $\pi_{\rm ss}$ is equivariant  with respect to  automorphisms of $C$ and descends to a map among  objects defined over any characteristic zero field $K$.

To summarise the relation between the correspondences we use: if $r'=(r,N)$ is a Frobenius-semisimple Weil--Deligne representation over $K$, then $$\pi(r')\subset \pi_{\rm gen}(r'),$$  with equality if and only if $\pi(r')$ is generic, and the  supercuspidal support of both is $\pi_{\rm ss}(r)$. 

\subsubsection{Galois representations} Let $p$ be a prime different from the residue characteristic of $F$, $K$ a topological field extension of $\Q_p$, $\rho \colon G_F\to \GL_n(K)$ a continuous representation admitting a Weil--Deligne representation $r'=(r, N)$. Then  for $*={\rm gen}, {\rm ss}, {\emptyset}$, we define
\begin{align}\label{pirho}
\pi_*(\rho):=\pi_*(r').
\end{align}

\section{Langlands correspondence for the Bernstein varieties}\label{sec: LL Ber}
We introduce the rational Bernstein variety, a scheme over $\Q$ whose ring of regular functions over $\C$ is the usual Bernstein centre of $G=G_{n}=\GL_{n}(F)$. Then we identify it with the coarse moduli space of semisimple representations of $W_{F}$. Finally, we prove an analogous result for Weil--Deligne representations with locally constant monodromy, in terms of an extension of the Bernstein variety.
\subsection{The Bernstein variety} 
Let $C$ be a field of characteristic zero. The Bernstein centre $\frak Z_{n,C}$ of $G$ is the centre (endomorphism ring of the identity functor) of the category of smooth $G$-modules with coefficients in $C$.  Bernstein and Deligne \cite{BD} gave an explicit description of it when $C$ is algebraically closed (in their original work $C=\C$ is assumed, but their arguments remain valid without this assumption): $\frak Z_{n, C}$ is the ring of regular functions of the `Bernstein variety' ${\frak X}_{n,C} $, an infinite  disjoint union of  affine varieties over $C$ which we now turn to describing.

 An \emph{inertial class} of supercuspidal supports over $C$  is an equivalence  class  $[M, \frak s]$ (which we sometimes denote just by $[\frak s]$) of pairs where $M$ is a Levi  and $\frak s$ is a  
 $\widehat{T}_{M}(C)$-orbit 
  of   supercuspidal representations of $M$ over $C$, up to $G$-conjugation of both. Here 
  $\widehat{T}_{M}(C)$
  is the space of unramified characters of $M$ valued in $C^{\times}$, which is the set of $C$-points of a split torus $\widehat{T}_{M}$ described as follows. Let $M_{0}\subset M$ be the intersection of the kernels of all unramified characters of $M$ (it    is also the subgroup generated by all compact subgroups of $M$); then $M/M_{0}$ is a free abelian group of finite rank and  $$\widehat{T}_{M}=\Spec C[M/M_{0}].$$
 
Let $H_{\frak s}$ be the (finite) group of unramified characters of $M$ stabilising any (hence every) $\sigma $ in $\frak s$.    As  $\frak s$ is a principal homogeneous space for $\widehat{T}_{M}/{ H_{\frak s}}$, it also acquires the structure of an  affine algebraic variety. 
 Let $W(M):=N_{G}(M)/M$ and $W(M, \frak s):={\rm Stab}_{W(M)}(\frak s)$, then 
$$\frak X_{[ \frak s]}:={\frak s}/W(M,\frak s)$$
is an algebraic variety whose points are in bijection with supercuspidal supports in the inertial class $[M, \frak s]$.  The Bernstein `variety'  is
$$\frak X_{n, C}=\coprod \frak X_{[\frak s]}$$
where $[M, \frak s]$ runs over inertial classes. 
The points of  $\frak X_{n , C}$ are in bijection with   supercuspidal supports over $C$.

\begin{lemm}  Let $C$ be any algebraically closed field of characteristic zero. Every inertial class of $G$ over $C$ contains an element defined over a finite extension of $\Q$ contained in $C$. 
\end{lemm}
\begin{proof}  By the theory of types  \cite{BK}, every irreducible supercuspidal representation $\sg$ of $G_{d}$ is compactly induced from a finite-dimensional representation $\sg_{0}$ of a compact-modulo-centre subgroup. After possibly twisting by a character, $\sg_{0}$ has a  model ove a finite extension of $\Q$, hence so does $\sg$.
\end{proof}

\begin{prop} There is a scheme $\frak X_{n}$, locally of finite type over $\Q$,  such that for any field $K$ of characteristic zero:
\begin{itemize}
\item  $\frak X_{n}(K)$ is   in bijection with $\Gal(\baar{K}/K)$-orbits of  supercuspidal supports over $\baar K$. 
\item if $C$ is algebraically closed, the scheme $ \frak X_{n,C}:= \frak X_{n}\times_{\Spec \Q}\Spec C$  is identified with the  previously defined one;
\item the centre $\frak Z_{n,K}$  of the category of smooth representations of $G_{n}$ over $K$-vector spaces is naturally identified with $\OO(\frak X_{n, K})$.
\end{itemize}
\end{prop}
\begin{proof} 
By the previous lemma,
the  orbits of the  obvious    Galois action on $\frak X_{n, \baar{\Q}}$  are finite; moreover if $x\in \frak X_{n, \baar\Q}$ then the Galois orbit of the connected component of $x$ is the union of the connected components of the Galois-conjugates of $x$. It follows that each Galois orbit  of connected components $\frak X^{(i)}_{n, \baar \Q}$ is an affine scheme of finite type over $\baar \Q$. Then the Galois action descends each $\frak X^{(i)}_{n, \baar \Q}$  to an algebraic variety  $\frak X_{n}^{(i)}$ over $\Q$ and the whole $\frak X_{n, \baar\Q}$ to a scheme $\frak X_{n}=\coprod_{i}\frak X_{n}^{(i)}$ over $\Q$.  By the previous lemma, if $C$ is any algebraically closed field then $\frak X_{n, \baar{\Q}}\ts_{\Spec\baar{\Q}} \Spec C =\frak X_{n, C}$.

Finally,  we give the isomorphism
 $$\frak Z_{n, K} \to \OO( \frak X_{n, K})= \frak Z_{n, \baar{K}}^{\Gal(\baar{K}/K)}.$$ 
 Let $f\in\frak Z_{n, K}$, then $f$ corresponds to a collection $f_{V}\in \End(V)$ for each smooth representation $V$ of $G$ over $K$, such that 
\beq\lb{centre prop}
f_{W}\psi=\psi f_{V}\eeq for each $\psi\in \Hom_{K[G]}(V, W)$. The Galois action on $\frak Z_{n, \baar{K}} $
 is given by $(f^{\tau})_{V}:= \psi_{\tau}^{-1}f_{V^{\tau}}$, where $V^{\tau}= V\ot_{\baar{K}, \tau}\baar{K}$ and $\psi_{\tau}\colon \End(V)\to \End(V^{\tau})$ is induced by the $K$-isomorphism $V\to V^{\tau}$.  By the forgetful functor from representations over $\baar{K}$ to representations over $\baar{K}$, we obtain a map $\frak Z_{n, K} \to \frak Z_{n, \baar{K}}$. By the property \eqref{centre prop}, its image consists of $\Gal(\baar{K}/K)$-invariant elements. Conversely if $f'\in  \frak Z_{n, \baar{K}}$ is Galois invariant, the elements $f_{V}:=f'_{V\ot_{K}\baar{K}} \in \End(V\ot\baar{K})^{\Gal(\baar{K}/K)}=  \End(V)$ form an element of $\frak Z_{n, K}$.  
\end{proof}

\subsubsection{Inertial local Langlands correspondence} Let $C$ be an algebraically closed field of characteristic zero. Define an \emph{inertial type} for $F$ over $C$ to be a representation of $I_{F}$ over $C$ which extends to $W_{F}$. Then, the \emph{inertial}  local Langlands correspondence $\pi_{I}$   is a bijection between inertial classes over $C$ on the $G$-side and inertial types over $C$ on the Galois side: $\pi_{I}^{-1}$ sends the inertial class $[M, \frak s]$ of a supercuspidal support $(M, \sigma)$ to the restriction of $\pi_{\rm ss}^{-1}((M, \sigma))$ to $I_{F}$.

\subsection{Semisimple local Langlands correspondence in families}
\label{sec: W-rep}
The following result is essentially due to Helm \cite{helm-curtis}. It morally  shows that the interpolation of the correspondence  $\pi_{\rm ss} $ of \eqref{piss} in families is trivial in the sense that there is a common moduli space parametrising the isomorphism classes of objects on both sides.  

We  denote by the same name a scheme and its functor of points.
\begin{theo}\label{coarse}
Let $K$ be a field of characteristic zero. The Bernstein variety $\frak X_{n, K}$ is a coarse (pro)-moduli scheme for the functor $\Phi_{n, K}$ which associates to any reduced   Noetherian $K$-scheme $X $
the set of isomorphism classes of semisimple representations of $W_{F}$ on locally  free $\OO_{X}$-modules of rank $n$, in such a way that the map 
$$\Phi_{n, K}\to \frak X_{n, K}$$
 induces $\pi_{\rm ss}^{-1}$ on geometric points.
\end{theo}

The proof will occupy the rest of this subsection. 

It suffices to show the result for $K$ algebraically closed, since by the compatibility with $\pi_{\rm ss}$ and the rationality properties of $\pi_{\rm ss}$ the map $\Phi_{n, \baar K}\to \frak X_{n,\baar K}$ descends to $K$. We thus assume $K$ is algebraically closed for the rest of this subsection and rename $K=C$ for a more suggestive notation.

Let $R_{n}^{\square}$ be the coordinate ring of the variety of semisimple matrices in $\GL_{n}(C)$.  
Let $R_{n}\into R_{n}^{\square}$ 
be the coordinate ring of the scheme parametrising conjugacy classes of semisimple automorphisms of an $n$-dimensional vector space; we have $R_{n}=C[a_{1}, \ldots , a_{n-1},a_{n}^{\pm 1}]$ where the $a_{i}$ correspond to the coefficients of the characteristic polynomial of an automorphism; we also have $\Spec R_{n}\cong ({\bf G}_{m}^{n})/{S_{n}}$, where the isomorphism $C[a_{1}, \ldots , a_{n-1},a_{n}^{\pm 1}]\cong C[x_{1}^{\pm 1}, \ldots, x_{n}^{\pm 1}]^{S_{n}}$ is given by $\prod_{i=1}^{n}(X-x_{i})=X^{n}+\sum_{i=1}^{n}a_{i}X^{i}$.

Let us start by making explicit Bernstein--Deligne's description of $\frak X_{n, C}=\coprod \frak X_{[ \frak s]}$. 
 We may fix  a supercuspidal support $(M, \sigma^{\circ})$ in the class $[M,\frak s]$ such that 
\begin{align}
\label{decsig}
\sigma^{\circ}=\otimes_{i=1}^{s}\sigma_{i}^{\otimes m_{i}}
\end{align}
where the $\sigma_{i}$ are pairwise  inertially inequivalent representations of $\GL_{d_{i}}(F)$ and $\sum m_{i}d_{i}=n$. Then $\frak X_{[ \frak s]}$ is explicitly described as follows. For $x\in {\bf G}_{m}$ denote by $\chi_{x,d}$ the unramified character of $\GL_{d}(F)$ given by $g\mapsto x^{v_{F}(\det(g))}$. Let $f_{i}\in \N$ be such that the group of $f_{i}^{\rm th}$ roots of unity 
\begin{align}\label{mufi}
\mu_{f_{i}}\subset {\bf G}_{m}\end{align} is the stabiliser of $\sigma_{i}$ under the action $x.\sigma_{i}= \sigma_{i} \chi_{x,d_{i}}$. Then, since $C $ is algebraically closed
\begin{align}\label{unif bernstein}
\frak X_{[ \frak s]}\cong \prod_{i=1}^{s} {\bf G}_{m}^{m_{i}}/{S_{m_{i}}}
\end{align}
with the point which is the orbit of $((x_{1,1}, \ldots, x_{1,m_{1}}), \ldots, (x_{s,1},\ldots, x_{s, m_{s}}))$ corresponding to
 \begin{align}\label{sigx}\otimes_{i=1}^{s} \otimes_{j=1}^{m_{i}}\sigma_{i} \chi_{x_{i,j}}^{f_{i}}.
\end{align}

 Let $$\frak X_{[ \frak s]}^{\square}= \prod_{i=1}^{s}\Spec R_{m_{i}}^{\square} \
\to \frak X_{[ \frak s]}$$
 be given by the natural quotients $\Spec R_{m_{i}}^{\square}\to \Spec R_{m_{i}}=({\bf G}_{m}^{m_{i}})^{S_{i}}$ in each factor. We will construct a  $W_{F}$-representation over $\frak X_{[ \frak s]}^{\square}$, which will be universal in some sense to be specified. We start by studying families of $W_{F}$-representations.

\begin{lemm}\label{CHT}  Let $C$ be an algebraically closed field of characteristic zero, let $A$ be a $C$-algebra and let $M$ be a continuous  $A[{W}_{F}]$-module, finite and locally free as an $A$-module. Then there is an $A[{W}_{F}]$-module decomposition 
$$M=\bigoplus_{[\tau]} {\rm Ind}_{{W}_{\tau}}^{{W}_{F}} M_{\tau},$$
whose summands are uniquely determined. Here
 $[\tau] $ runs over the ${W}_{F}$-conjugacy classes   of irreducible ${I}_{F}$-representations occurring in $M$,   ${W}_{\tau}\subset {W}_{F}$ is the stabiliser of a chosen representative $\tau\in [\tau]$, and 
\begin{align}\label{Mtau}
M_{\tau}=\Hom_{{I}_{F}}(\tau, M)\otimes \tau.
\end{align}
\end{lemm}
\begin{proof} 
By the finiteness of $M$, there is a finite index subgroup ${I}_{F}^{\circ}\subset {I}_{F}$ acting trivially on $M$; let $\baar{{I}}_{F}:={I}_{F}/{I}_{F}^{\circ}$. We have  $A[\baar{{I}}_{F}]\cong C[\baar{{I}}_{F}]\otimes A$ and the finite-group algebra $C[\baar{{I}}_{F}]\cong \oplus_{\tau} \End(\tau)$ is semisimple. This induces a decomposition $M=\oplus_{\tau}M_{\tau}$ as $A[{I}_{F}]$-modules, where $\tau$ runs over equivalence classes of irreducible $\baar{{I}}_{F}$-representations and $M_{\tau}=\Hom_{{I}_{F}}(\tau, M)\otimes \tau$. 

By definition,  the finite cyclic group ${W}_{F}/{W}_{\tau}$ acts on $\bigoplus_{\tau\in [\tau]}M_{\tau}$ by permuting the summands. The description of $M_{\tau }$ follows.
\end{proof}

\subsubsection{Universal inertially-framed representation} 
Fix a representative $\tau$ for every ${{W}_{F}}$-conjugacy class $[\tau]$ of irreducible ${{I}_{F}}$-representations. An \emph{inertial framing}\footnote{Analogous to the notion of  \emph{pseudo-framing} in \cite{helm-curtis}.} of an $A$-representation $r$ of ${I}_{F}$ or ${W}_{F}$ is a choice, for every $[\tau]$, of a basis of $\Hom_{{I}_{F}}(\tau,r)$.

Let  $[M, \frak s]$ be the inertial class fixed above and let   $\nu=\pi_{I}^{-1}([M, \frak s])$. We shall construct an inertially-framed semisimple $W_{F}$-representation $r_{\nu}$ over $\frak X_{[ \frak s]}^{\square}$ which is universal among those such representation whose restriction to $I_{F}$ is $\nu$. 

  Decompose $\nu=\bigoplus_{\tau}\Hom_{{I}_{F}}(\tau, \nu)\otimes \tau$, and let $n_{[\tau]}:=\dim \Hom(\tau, \nu)$, which is independent of $\tau\in [\tau]$. If we 
fix  a supercuspidal support $(M, \sigma^{\circ})$ as in \eqref{decsig},
we obtain a representation $r^{\circ}=\pi_{\rm ss}^{-1}((M, \sigma^{\circ}))$ of $W_{F}$ described as follows. 

Let $\wtil\tau_{i}^{F}$ be the irreducible $d_{i}$-dimensional representation of $W_{F}$ associated with  $\sigma_{i}$ by the local Langlands correspondence for $\GL_{d_{i}}(F)$. By the compatibility of local Langlands with twisting, $\mu_{f_{i}}$ \eqref{mufi} is also the stabiliser of $\wtil \tau_{i}^{F}$ under the action of ${\bf G}_{m}$ by unramified twists. It follows that we can write $\wtil\tau_{i}^{F}={\rm Ind}_{W_{i}}^{W_{F}}\wtil \tau_{i}$ where $W_{i}$ is the Weil group of the unramified extension of $F$ of degree $f_{i}$ and $\wtil \tau_{i}$ is a representation of $W_{i}$ whose restriction to $I_{F}$ is irreducible. Then 
$$r^{\circ}= \bigoplus_{i}(\wtil\tau_{i}^{F})^{m_{i}}=\bigoplus_{i}{\rm Ind}_{W_{i}}^{W_{F}}(C^{m_{i}}\otimes\wtil\tau_{i})$$
where $C^{m_{i}}$ is  a trivial $W_{i}$-module. 

Define the representation 
\begin{align}\label{rnu}
r_{\nu}:=\bigoplus_{i}{\rm Ind}_{W_{i}}^{W_{F}}((R_{m_{i}}^{\square})^{m_{i}}\otimes \wtil\tau_{i})
\end{align}
over $\frak X_{[ \frak s]}^{\square}=\Spec \bigotimes_{i} R_{m_{i}}^{\square}$,
where each $(R_{m_{i}}^{\square})^{m_{i}}$ is an unramified $W_{i}$-module on which a geometric Frobenius $\phi_{i}$ acts by the universal semisimple $m_{i}\times m_{i}$-matrix over $R_{m_{i}}^{\square}$. It is clear that $r_{\nu}$ is universal among those $W_{F}$-representations $r$ such that $r|_{I_{F}} \cong \nu$.

We can now establish the coarse moduli property of $\frak X_{n, C}$.  Recall that this amounts to showing that: 
\begin{itemize}
\item there is a natural transformation of functors $\Phi_{n,C}\to \frak X_{n, C}$, that is: for every isomorphism class   of rank-$n$ families of representations of $W_{F}$ over  a  scheme $X$ in ${\rm Noeth}_{X}$, there is a map $f\colon X\to \frak X_{n, C}$, functorially in $X$;
\item for each algebraically closed field $\Omega \supset C$, the above transformation  induces a bijection $\Phi_{n}(\Omega)\to \frak X_{n, C}(\Omega)$  (and moreover we assert that this bijection is $\pi_{\rm ss}^{-1}$).
\end{itemize}
First, for each inertial type $\nu$, let $\Phi_{\nu}\subset \Phi_{n}$ be the subfunctor of those classes of representations $r$ such that $r|_{I_{F}}$ is everywhere equivalent to $\nu$. Suppose that $r$ is a rank-$n$ family of representations of $W_{F}$ over a reduced Noetherian scheme $X$. As the restriction $r|_{I_{F}}$ is locally constant on $X$ and $X$ has finitely many connected components, $r$ defines an element of a finite union $\coprod_{\nu}\Phi_{\nu}(X)\subset \Phi(X)$. Cover $X$ with connected Zariski-open subsets $\iota_{i}\colon X_{i}\into X$ such that over each  $X_{i}$ the representation $\iota_{i}^{*}r|_{I_{F}}$ is isomorphic to a fixed inertial type $\nu$ and moreover it admits an inertial framing. By the universal property of $r_{\nu}$  there is a unique  map $f_{i}\colon X_{i}\to \frak X_{n}^{\square}$ such that $f_{i}^{*}r_{\nu}\cong \iota_{i}^{*}r$. The compositions of $f_{i}$ with the projection $ \frak X_{n}^{\square}\to  \frak X_{n, C}$ are independent of the choice of inertial framing and glue to the desired  map $f\colon X\to \frak X_{n, C}$. 

Finally, we show that when $X=\Spec \Omega$, the image of the supercuspidal support $x=f(X) \in\frak X_{n, C}(\Omega)$ under $\pi_{\rm ss}$ is $r$. In fact, suppose that $x$ is represented by the orbit of 
 $$((x_{1,1}, \ldots, x_{1,m_{1}}), \ldots , (x_{s,1},\ldots, x_{s, m_{s}}))\in \prod_{i=1}^{s}{\bf G}_{m}^{m_{i}},$$ and keep the notation introduced before. Then for any choice of inertial framing   the $i^{\rm th}$ factor of $f^{*}r_{\nu}$ is the tensor product of $\wtil \tau_{i}^{F} $ and an $m_{i}$-dimensional unramified representation over $\Omega $ on which  $\phi=\phi_{i}^{f_{i}}$ acts on by a matrix with eigenvalues $x_{i, 1}^{f_{i}}, \ldots, x_{i, m_{i}}^{f_{i}}$. By the compatibility of $\pi_{\rm ss}$ with twisting, this is the $W_{F}$-representation corresponding under $\pi_{\rm ss}$ to the supercuspidal support $x$ as described in \eqref{sigx}.

This completes the proof of Theorem \ref{coarse}.

\subsection{Monodromy and the extended Bernstein variety}\label{sec: ext}
We start by discussing monodromy operators in families, mostly borrowing from \cite[\S 7.8.1]{BCh}.

\subsubsection{Partitions}  A partition of length $\ell$ of an integer $m$ is an ordered sequence of integers $t=(t_{j})$ such that $n=\sum_{j=1}^{\ell} t_{j} $. Two partitions $t$, $t'$ are said to be equivalent if there is a  $\sigma\in S_{\ell}$ such that $t_{\sigma(j)}=t_{j}$ for all $j$; the subgroup of self-equivalences of a partition $t$ is denoted by $W_{t}\subset S_{\ell}$. 
 A \emph{multipartition} is a sequence $t=(t_{i})$ of   partitions $m_{i}=\sum_{j=1}^{\ell_{i}} t_{i,j}$; two multipartitions $t=(t_{i})$, $t'=(t_{i}')$ are said to be equivalent if $t_{i}$ is equivalent to $t_{i}'$ for all $i$; the group of self-equivalences of $t=(t_{i})$ is  $W_{t}=\prod_{i}W_{t_{i}}$. If $t$ is a  (multi)partition, we denote by $[t]$ its equivalence class; if $[t] $ is class of (multi)partitions, a decreasing   representative $t$ is one satisfying  $t_{(i),j}\geq t_{(i),j+1}$ for all ($i$ and) $j$.

\subsubsection{Nilpotent operators} Let $K $ be a field and let $N$ be a  nilpotent operator on a vector space $V$ of dimension $m$. Then $N$ can be decomposed into Jordan blocks,  $N\sim J_{t_{1}}\oplus \ldots \oplus J_{t_{\ell}}$ for  a partition $t_{1}+\ldots t_{\ell}=m$ whose class  $$[t]=[t](N)$$ is uniquely determined.
Let now $N'$ be another nilpotent operator on a $K$-vector space $V'$ of the same dimension $m$, with associated partition $[t']$. We write 
\begin{align}\label{<mon}
N\preceq N', \quad \textrm{resp. } N\sim N' 
\end{align}
if $[t]\preceq [t']$ (resp. $[t]=[t']$) in the sense that, if $t$ and $t'$ are decreasing representatives of $[t]$ and $[t']$ then for all $i$, 
$$\sum_{j=1}^{i} t_{j} \leq \sum_{j=1}^{i} t_{j}';$$
equivalently, if ${\rm rk} N^{i}\leq {\rm rk} N'^{i}$ (resp.  ${\rm rk} N^{i} =  {\rm rk} N'^{i}$) for all $i\leq m$.  The association $N\mapsto [t]$, hence the relations $\preceq$, $\sim$, are compatible with field extensions.

\subsubsection{Semicontinuity of monodromy} Let $K$ be a characteristic zero field, $X$ a reduced Noetherian scheme over $K$, and  let $(r,N)$ be a Weil--Deligne representation on a locally free $\OO_{X}$-module $M$ of rank $n$. As in the proof of Lemma \ref{CHT}, there exists a finite set of  irreducible representations $\tau$ of $I_{F}$ over $\baar{K}$ such that $M\otimes \baar{K}=\bigoplus \tau\otimes \Hom_{I_{F}}(\tau, M)$, and  $\Hom_{I_{F}}(\tau, M)$ is locally free. The monodromy operator $N$ leaves invariant the summands and, as it commutes with $I_{F}$, induces nilpotent operators $N_{\tau}$ on each of $\Hom_{I_{F}}(\tau, M)$. 

Let $x$, $y$ be points of $X$ and let $N_{x}=(N_{\tau, x})_{\tau}$, $N_{y}=(N_{\tau, y})_{\tau}$ be the associated sequences  of associated nilpotent operators defined above. We say that $N_{x}\preceq_{I_{F}} N_{y} $  if for all $\tau$, $N_{x,\tau}  \preceq N_{y, \tau}$, and that $N_{x}\sim_{I_{F}}N_{y}$ if for all $\tau$, $N_{\tau,x}\sim N_{\tau, y}$.  We say that $N$ is \emph{locally constant} on $X$ if $N_{x}\sim_{I_{F}} N_{y}$ whenever $x$ and $y$ belong to  the same connected component of $X$.

\begin{prop}\label{nowhere dense} The map $${\rm rk}\, N\colon x\mapsto {\rm rk}\, N_{x}:=({\rm rk}\, N^{i}_{x,\tau})_{\tau, i} $$ is lower semicontinuous in the sense that, for each $x\in X$, there is an open neighbourhood $U\subset X$ of $x$ such that $N_{x} \preceq_{I_{F}} N_{y}$ for all $y\in U$. Moreover the set $D$ of discontinuity points of ${\rm rk} \, N$ is closed and nowhere dense in $X$, and it contains no $x\in X$ such that $(r_{x}, N_{x})$ is pure.

The restriction of $(r,N)$ to the open dense $X-D$ has locally constant monodromy.
\end{prop}
(We say that ${\rm rk} N$ is discontinuous at $x\in X$ if for every open neighbourhood $U\subset X$ of $x$ there is $y\in U$ such that ${\rm rk }\, N_{x}\nsim_{I_{F}} {\rm rk }\,N_{y}$. )
\begin{proof}
The lower semicontinuity of $ {\rm rk}\, N$ follows from the upper semicontinuity of the ranks of the coherent $\OO_{X}$-modules $M_{ \tau, i}:=\Ker N_{\tau}^{i}$ for all $i$, $\tau$.

  For the second assertion, we first note that the set $D_{\tau, i}$ of discontinuity points of the function ${\rm r}:={\rm rk}\, M_{\tau, i}$ is the union, for all irreducible components $X_{j}\subset X$, of the sets $D_{\tau, i;j}$ of discontinuity points of ${\rm r}_{j}:={\rm r}|_{X_{j}}$. If $r_{j}$ is the rank of $M_{\tau,i}$ at the generic point of $X_{i}$, each of $D_{\tau, i; j}=\{{\rm r}_{i}(x)\leq r_{i}-1\}$ is  closed in $X_{j}$ hence in $X$. Therefore $D=\bigcup_{i, \tau; j} D_{i, \tau;j}$ is closed in $X$ and, as it contains no generic points of irreducible components, it is nowhere dense. By \cite[Theorem 3.1 (2)]{saha},\footnote{In this statement, the local field $F$ is assumed to be of characteristic zero, however this is not used in the proof.}
 $D$ contains no $x\in X$ such that  $(r_{x}, N_{x})$ is pure.
  Finally, the restriction of ${\rm rk}\, N$ to $X-D$ is continuous, equivalently the monodromy is  locally constant. 
\end{proof}

The partial order on monodromy operators has a Langlands correspondent.
\begin{prop}[Zelevisnky] \lb{ss'}
Let $(r, N) $ and $(r, N')$ be $n$-dimensional Frobenius-semisimple representations of $W'_{F}$ over a characteristic-zero field $K$, with isomorphic restriction to $W_{F}$.  Suppose that $N'\preceq N$. Then 
$$\Hom(\pi_{\rm gen}((r , N')) , \pi_{\rm gen}((r, N)))
$$
is $1$-dimensional over $K$, generated by a surjection 
which is an isomorphism if and only if $N'\sim N$.
\end{prop} 
\begin{proof}
This follows from the definition of $\pi_{\rm gen}$ and \cite[Theorem 4.5.1, Propositions 4.3.6, 4.3.7]{EH}. Note that if ${\bf s}$ is a multisegment, in \emph{loc. cit.} the  notation  $\pi({\bf s}) $ corresponds to our $\pi'_{\rm gen}({\bf s})$ from \S\ref{21}, so that their morphisms go in the opposite direction to ours.
\end{proof}

\subsubsection{The extended Bernstein variety} We now define an extension of the Bernstein variety whose points are in bijection with Galois orbits of multisegments for $G$.
 It is naturally a coarse moduli space for Weil--Deligne representations with locally constant monodromy.

  Fix an inertial   class $[\frak s]$, which we can write in the form
  \begin{align}\label{fraks d}
  \frak s=\otimes_{i=1}^{s} \frak s_{i}^{\otimes m_{i}}
  \end{align}
where the $\frak s_{i}$ are pairwise inequivalent  inertial classes of supercuspidal representations of $G_{d_{i}}(F)$.

 Let $[t]=([t_{i}])$ be a multipartition of $(m_{i})$ and, after picking any $t\in [t]$ (whose choice won't matter),  let 
$$ \frak X_{[\frak s], [t]}:=\big(\prod_{i=1}^{s}  \frak  X_{[\frak s_{i}]}^{\ell_{i}}\big)/W_{t},$$
a scheme  over $\baar\Q$ with a closed immersion 
\begin{align*}
f_{[t]}\colon  \frak X_{[\frak s],[t]} &\to \frak X_{[ \frak s]}\\
(\sigma_{1, 1}, \ldots, \sigma_{1, \ell_{1}}, \ldots, \sigma_{s,1}\ldots, \sigma_{s, \ell_{s}})
&\mapsto \otimes_{i=1}^{s}\otimes_{j=1}^{\ell_{i}}  \Delta(\sigma_{i,j}, t_{i,j})   
\end{align*}
 where $\Delta(\sigma, t)  := \otimes_{m=0}^{t-1}  \sigma(m)$   
When each  $t_{i}=1+\ldots+1 $, we have $\ell_{i}=m_{i}$ and  the above map recovers the isomorphism \eqref{unif bernstein}.

The  action of $\tau\in \Gal(\baar\Q/\Q)$  on $\frak X_{n, \baar\Q}$ sends $\frak X_{\frak s, [t]}$ to $\frak X_{\frak s^{\tau},[t]}$ and the (finite) union $\coprod_{\tau} \frak X_{\frak s^{\tau}, [t]}$ descends to a $\Q$-subscheme of $\frak X_{n}$. Then 
$$\frak X_{n, \baar\Q}' := \coprod_{[\frak s], [t]} \frak X_{[\frak s], [t]}$$
descends to a scheme
$$\frak X_{n}'$$
over $\Q$: the \emph{extended Bernstein variety}.

\begin{theo}\label{coarseWD}The extended Bernstein variety $\frak X'_{n, K}$ is a coarse (pro)-moduli scheme for the functor $\Phi_{n, K}'$ which associates to any reduced Noetherian $K$-scheme $X$
 the set of isomorphism classes of Frobenius-semisimple Weil--Deligne  representations  on locally  free $\OO_{X}$-modules of rank $n$ with locally constant monodromy, in such a way that the map $\Phi_{n,K}'\to\frak X_{n,K}'$ induces the local Langlands correspondence $\pi^{-1}$ on  geometric points.
This map fits into a commutative diagram
\begin{equation*}\label{quotient}
\xymatrix{
\Phi_{n, K}'\ar[r]\ar[d] &\frak X_{n, K}'\ar[d]\\
\Phi_{n, K}\ar[r] & \frak X_{n, K}
}
\end{equation*}
where the bottom horizontal arrow is given by Theorem \ref{coarse}.
\end{theo}
\begin{proof} We may reduce to the case $K=C$ is algebraically closed and  to testing the coarse moduli property on connected schemes $X_{/C}$. But it is easy to see that if  $(r, N)$ is a Weil--Deligne representation with constant monodromy over $X$ and the map $f\colon X\to \frak X_{n, C}$ induced by $r$ via Theorem \ref{coarse} has image   in $\frak X_{[\frak s]}$, then in fact $f$ has values in $\frak X_{[\frak s], [t]}$ for $[t]=[t](N)$, so that we may uniquely lift $f$ to $f'\colon X\to \frak X_{[\frak s], [t]}\subset \frak X'_{n, C}$. Moreover the isomorphism classes of Weil--Deligne representations over $C$ are   in bijection with points of $\frak X'_{n, C}$: explicitly, the point of $\frak X_{[\frak s], [t]}\subset \frak X_{n, C}'$ corresponding  to the supercuspidal support $\otimes_{i=1}^{s}\otimes_{j=1}^{\ell_{i}} \Delta(\sigma_{i}, t_{i,j})$ corresponds to the Weil--Deligne representation 
\begin{align}\label{speh}
\oplus_{i=1}^{s}\oplus_{j=1}^{\ell_{i}}{\rm Sp}(\pi_{\rm ss}^{-1}(\sigma_{i}), t_{i,j}).
\end{align}
\end{proof}

Let $\frak X_{[\frak s], [t]}^{\square}:=\prod_{i=1}^{s}\frak X_{[\frak s_{i}]}^{\square}\to \frak X_{[\frak s],[t]}$. There is a `universal' Weil--Deligne representation $(r, N)_{[\frak s], [t]}$ over $  \frak X_{[\frak s], [t]}^{\square}$ given by 
\begin{align}\label{univ wd}
(r, N)_{[\frak s], [t]}:= \bigoplus_{i} \bigoplus_{j}
  {\rm Ind}_{W_{i}'}^{W_{F}'}( (R_{\ell_{i}}^{\square})^{\ell_{i}}\otimes{\rm Sp}( \wtil{\tau}_{i}, t_{i,j}))
\end{align}
with the notation of \eqref{rnu}. Its fibre over any point of $\frak X_{\frak s, [t]}^{\square}$ over the supercuspidal support  $\otimes_{i=1}^{s}\otimes_{j=1}^{\ell_{i}}\Delta(\sigma_{i}, t_{i,j})$ is isomorphic to \eqref{speh}.

\section{Langlands correspondence in families}\label{sec: LLF}

\subsection{The Bernstein--Zelevinsky top derivative}\label{s4.1}
 Throughout this subsection,  $K$ denotes   a field of characteristic zero and  $X$ a locally Noetherian $K$-scheme unless otherwise noted.
 \paragraph{Smooth families of representations} Let $G$ be a locally compact and totally disconnected group. An $\OO_{X}[G]$-module $V$ is said to be \emph{smooth} if every  $v\in V$ is stabilised by some compact open subgroup of $G$, and \emph{admissible} if for every compact open subgroup $U\subset G$, the $\OO_{X}$-module $V^{U}$ is coherent (that is, finitely generated locally on $X$). We say that $V$ is \emph{finitely generated} as $\OO_{X}[G]$-module if for  every $x\in X$ there exist a neighbourhood $X'$, a positive integer $n$, and an $\OO_{X'}[G]$-surjection  $\OO_{X'}[G]^{\oplus n}\to V_{|X'}$. 
\begin{lemm}\lb{is coh} Let $V$, $W$ be $\OO_{X}[G]$-modules. Assume that $V$ is finitely generated and $W$ is admissible. Then $\Hom_{\OO_{X}[G]}(V, W)$ is a coherent $\OO_{X}$-module.
\end{lemm} 
\begin{proof}  Over any sufficiently small open  $X'\subset X$, $V$ is generated by finitely many elements $v_{i}$, $i\in I$. Let $U\subset G$ be an open compact subgroup such that $v_{i}\in V^{U}$ for all $i\in I$. Then $f\mapsto (f(v_{i})_{i\in I}$ defines an injection $\Hom_{\OO_{X'}[G]}(V_{|X'}, W_{|X'})\into (W/_{|X'}^{U})^{I}$. 
\end{proof}

\paragraph{Notation} For the rest of this subsection, $F$ denotes a global field, and $S$ a finite set of non-archimedean places of $F$.  

Let $G_{v}=G_{n,v}:=\GL_{n}(F_{v})$. We denote by  $N_{n,v}\subset P_{n,v}\subset G_{n,v}$ respectively the upper-triangular unipotent subgroup and the mirabolic subgroup consisting of matrices with last row $e_{n}^{\rm t}$.  We denote
$N_{n,S}:=\prod_{v\in S} N_{n,v}\subset P_{n,S}:= \prod_{v\in S}P_{n,v}\subset   G_{n,S} :=\prod_{v\in S} G_{n,v}$. All subscripts $n$ will be omitted when clear from context.

\paragraph{Additive characters} We recall some formalism introduced in \cite{pyzz}. Let  $\bmu_{\Q}$ be the ind-scheme of roots of unity over $\Q$. We define the space of additive characters of $F_{v}$ of level ${0}$ to be  
\begin{align}\label{psiuniv} \Psi_{v}:=\Hom(F_{v}/\OO_{F,v}, \bmu_{\Q}) -  \Hom(F_{v}/\vpi_{v}^{-1}\OO_{F,v}, \bmu_{\Q}), 
\end{align}
where $\vpi_{v}\in F_{v}$  is a uniformiser and we regard $\Hom(F_{v}/\OO_{F,v}, \bmu_{\Q})$ as a profinite group scheme over $\Q$.\footnote{If $F_{v}=\Q_{\ell}$, then  $\Hom(F_{v}/\OO_{F,v}, \bmu_{\Q})=T_{\ell}\bmu_{\Q}$, the $\ell$-adic Tate module of roots of unity.} The scheme $\Psi_{v}$ is a torsor for the action of $\OO_{F,v}^{\times}$ (viewed as a constant profinite group scheme over $\Q$) by $a.\psi_{v}(x):=\psi_{v}(ax)$. 
 We denote by $\psi_{{\rm univ},v}\colon F_{v}\to \OO(\Psi_{v})^{\times}$ the universal additive character of level $0 $ of $F_{v}$. 
 
\paragraph{The top derivative functor}
Extend any additive character $\psi_{v}$ of $F_{v}$  to a character of $N_{v}$, still denoted by $\psi_{v}$, by $u\mapsto \psi_{v}(u_{1,2}+\ldots+u_{n-1,n})$.

\begin{prop}  \lb{topder} There is a  functor, called the Bernstein--Zelevinsky top derivative,
$$V\mapsto V^{(n)}$$
from smooth $\OO_{X}[G_{v}]$-module to  $\OO_{X}$-modules,
satisfying:
\begin{enumerate}
\item if $X\stackrel{f}\to \Psi_{v}$ is a $\Psi_{v}$-scheme, then 
$$V^{(n)}=V/V({N_{v}, \psi_{X,v}}),$$
where $\psi_{X,v}\colon N_{v}\stackrel{\psi_{{\rm univ},v}}{\longrightarrow}\OO(\Psi_{v})^{\ts}\stackrel{f^{*}}{\longrightarrow}\OO(X)^{\ts} $  is the composition and $V({N_{v}, \psi_{X,v}})= \{n.v-\psi_{X,v}(n)v\colon v\in V\}$;
\item it is exact in the sense that for any $\OO_{X}$-module $M$, $(V\ot_{\OO_{X}} M)^{(n)}= V^{(n)}\ot_{\OO_{X}} M$;
\item if $V$ is finitely generated then $V^{(n)}$ is coherent.
\end{enumerate}
\end{prop}
\begin{proof} The definition and basic properties of the top derivative, including exactness, are recalled in \cite[\S 3.1]{EH}: see in particular (the proof of)  Proposition 3.1.4 for the descent from $\Psi_{v}$-schemes to $\Q$-schemes.  We prove part 3, dropping all subscripts $v$ from the notation.  The action of the Bernstein centre gives a map $\OO_{X}\ot\OO_{\frak X}\to \End_{\OO_{X}[G_{}]}(V)$ whose image we denote by $A$: by Lemma \ref{is coh}, $A$ is a coherent sheaf of $\OO_{X}$-algebras. Then we may view $V$ as a (finitely generated) $\OO_{Y}[G]$-module where ${Y}={\underline{\Spec}}_{\OO_{X}}A$. By construction, the   Bernstein centre acts on $V$ by scalars through the map $ \OO_{\frak X}\to \OO_{Y}$. After a base-change we may assume that $Y$ is a $\C$-scheme. 
 Let $\psi_{}$ be any $\C^{\ts}$-valued character of $F_{}$.
By the adjunction of \cite[Proposition 3.1.3.2]{EH}, we have an isomorphism
$$V^{(n)}\cong \Hom_{\C[P_{}]} ( \text{c-Ind}_{N_{}}^{P_{}}\psi, V)=  \Hom_{\C[P_{}]} ( (\text{c-Ind}_{N_{}}^{P_{}}\psi) \ot  |\det|,
V \ot|\det|).$$
By the following  lemma, together with the general  compatibility $ \text{c-Ind}_{P_{}}^{G_{}} \text{c-Ind}_{N_{}}^{P_{}}= \text{c-Ind}_{N_{}}^{G_{}}$ and the well-known fact the the modulus character of $P$ is $|\det|^{-1}$, the latter space injects into 
$$\Hom_{\C[G_{}]} ( \text{c-Ind}_{N_{}}^{G_{}}\psi_{}, V\ot |\det|)= \Hom_{\OO_{Y}[G_{}]} ( (\text{c-Ind}_{N_{}}^{G_{}}\psi_{})\ot_{\OO_{\frak X_{\C}}}\OO_{Y}, V\ot |\det|),$$
where the identity holds because every homomorphism of smooth $\C[G_{}]$-modules commutes with the action of the Bernstein centre.
By a key result of Bushnell--Henniart \cite{BH}, the $\OO_{{\frak X}_{\C}}[G_{}]$-module  $\text{c-Ind}_{N_{}}^{G_{}}\psi$ is finitely generated, hence the $\Hom$-space under consideration is a coherent  $\OO_{Y}$-module, hence also a coherent $\OO_{X}$-module.
\end{proof}

\begin{lemm} Let $G$ be locally compact, totally disconnected topological group and let $P$ be a closed subgroup. Let $\delta_{P}$, $\delta_{G}$ be the respective modulus characters and let $\delta_{P\bks G}:=  \delta_{P}^{-1}\delta_{G|P}$. Let $W$ be a smooth $\C[P]$-module, $V$ be a smooth $\C[G]$-module. Then there is a natural injection  
\beq \Hom_{P}(W, V) \to \Hom_{G}( \textup{c-Ind}_{P}^{G}  (W \ot \delta_{P\bks G}) , V).
\eeq
\end{lemm}
I am grateful to D. Helm for suggesting  the statement of this lemma.
\begin{proof} By \cite[III.2.6]{renard}, there is a canonical isomorphism $\text{c-Ind}_{P}^{G}( W\ot \delta_{P\bks G})\cong \cH_{G}\ot_{\cH_{P}} W$. Viewing it  as an identification, the asserted injection  is  given by the map
$$\phi\mapsto ( \phi' \colon f* w\mapsto f*\phi(w)).$$
This is injective since, if $\phi\neq 0$ and  $w\in W $ is such that $\phi(w)\in V^{U}$ is nonzero, then $\phi'(e_{U}* w)= \phi(w)\neq 0$ for the idempotent $e_{U}\in \cH_{G}$ corresponding to $U$.
\end{proof}

If $V$ is an $\OO_{X}[G_{S}]$-module and $v\in S$, we denote by $(\ )^{(n),v}$ the top derivative with respect to $G_{v}$ and, if $T\subset S$, by $(\ )^{(n), T}$ the composition of $(\ )^{(n), v}$ for all $v\in T$. The space $V^{(n),T}$ is an $\OO_{X}[G_{S-T}]$-module. If $T=S$, we simply write
$$V^{(n)}:=V^{(n), S}, $$
an $\OO_{X}$-module.

\subsection{Co-Whittaker modules}  We continue with the notation of the previous subsection. By \cite[\S 3.1]{EH}, there is a functor on smooth $\OO_{X}[G_{S}]$-modules
 $$V\mapsto \mathfrak{J}(V)=\mathfrak{J}_{S}(V),$$ 
 the ``space of  Schwartz functions''  in $V$ with respect to all $v\in S$, which also commutes with base-change.  The space $ \mathfrak{J}(V)\subset V$ is the maximal $P_{S}$-stable $\OO_{X}$-submodule such that $\mathfrak{J}(V)^{(n)}=V^{(n)}$. 

\begin{lemm}\label{schwartz} For   a smooth admissible $\OO_{X}[G_{S}]$-module $V$,
 the following are equivalent:
\begin{enumerate}
\item   Every nonzero $\OO_{X}[G_{S}]$-module quotient $ W$ of $V$ is non-degenerate (i.e. it satisfies  $W^{(n)}\neq 0$).
\item The space $\mathfrak{J}(V)$ generates $V$ over $\OO_{X}[G_{S}]$.
\end{enumerate}
Moreover if $V$ is finitely generated over $\OO_{X}[G_{S}]$,  the set of $x\in X$ such that $V\otimes K(x)$ satisfies the above conditions is open.
\end{lemm}
As pointed out to us by the referee, the proof of the second part would fail in a characteristic dividing the pro-order of $G_{S}$. 
\begin{proof} We may reduce to the case where $X$ is the spectrum of a Noetherian local ring; then the equivalence is proved in \cite[Lemma 6.3.2]{EH}. The set of $x\in X$ such that $V\otimes K(x)$ satisfies both  conditions is    open as it is the complement of the support of the  coherent $\OO_{X}$-module $M$ defined as follows: let $U\subset G_{S}$ be an open compact subgroup such that $V^{U}$ generates $V$ over $\OO_{X}[G_{S}]$. Then $M$ is the sheaf of $U$-invariant vectors of the quotient of $V$ by the $\OO_{X}[G_{S}]$-span of $\mathfrak{J}(V)$.  
\end{proof}
  The following is \cite[Definition 2.1]{hm-conv} (after \cite[Definition 6.1]{helm-whitt}).
\begin{defi}\label{coW def} A smooth admissible  $\OO_{X}[G_{S}]$-module $V$ is said to be 
\begin{itemize}
\item  of \emph{Whittaker type} if    $V^{(n)}$ is  a locally free $\OO_{X}$-module of rank $1$;
\item \emph{strictly of Whittaker type} if $V^{(n)}\cong \OO_{X}$.
\item \emph{(strictly) co-Whittaker} if  it is (strictly) of Whittaker type and it satisfies the conditions of Lemma \ref{schwartz} (that is,  $\mathfrak{J}(V)$ generates $V$ over $\OO_{X}[G_{S}]$).
\end{itemize}  
\end{defi} 
\begin{lemm}\label{coW open} Let $V$ be a smooth    finitely generated $\OO_{X}[G_{S}]$-module such that $(V\otimes K(x))^{(n)}\neq 0$ for all $x$ in a dense subset $\Sigma \subset X$.
 The set of $x'\in X$ such that $V\otimes K(x')$ is co-Whittaker is open.
\end{lemm}
\begin{proof} 
By Proposition \ref{topder}.3, the $\OO_{X}$-module $V^{(n)}$ is coherent. By the semicontinuity of its rank, the set $U_{r}\subset X$ of those $x$ such that $V^{(n)}\otimes K(x)$ has rank $\leq r$ is open for all $r$; by assumption, $U_{0}\cap \Sigma=\emptyset$ hence $U_{0}=\emptyset $ and at all points of the open $U:=U_{1}$, $V^{(n)}$ has fibre-rank $1$ so that  $V^{(n)}|_{U}$ is an invertible sheaf. The openness of the other two conditions in Definition \ref{coW def} is Lemma \ref{schwartz}.
\end{proof}
\begin{defi}\label{sub-coW}  If $V$ is a finitely generated $\OO_{X}[G_{S}]$-module of Whittaker type, its \emph{maximal co-Whittaker submodule} $V'$ is the $\OO_{X}[G_{S}]$-span of $\frak J(V)$. \end{defi}
By the previous Lemma, a fibre $V_{x}$  is co-Whittaker if and only if $V'_{x}=V_{x}$, i.e. $x$ belongs to the open complement of the support of $V/V'$.

\begin{lemm}\label{schur} Let $V$ be a  torsion-free  co-Whittaker module over $X$. Then the natural map  $\OO_{X}\to\underline{\rm End}_{\OO_{X}[G]}V$ is an isomorphism.
\end{lemm}
\begin{proof} The argument of \cite[Proposition 6.3.4 (3)]{EH} carries over to our context. 
\end{proof}
\begin{lemm}\label{415} Let $V_{1}$, $V_{2}$ be torsion-free  co-Whittaker modules over $X$ and suppose that there is a  dense subset $\Sigma \subset X$ such that $$V_{1}\otimes K(x)\cong V_{2}\otimes K(x)$$ for all $x\in \Sigma$. 
Then $L:= \underline{\Hom}_{\OO_{X}[G_{S}]}(V_{2}, V_{1})$ is an invertible sheaf and  $$V_{1}\cong V_{2}\otimes L$$ as $ \OO_{X}[G_{S}]$-modules.
\end{lemm}
\begin{proof} We may  reduce   to the case where $X$ is affine, treated in 
 \cite[Lemma 6.3.7]{EH}. 
\end{proof} 

The previous result can be generalised.

\begin{lemm} 
 Let $V_{1}$, $V_{2}$ be smooth admissible torsion-free $\OO_{X}[G_{S}]$-modules over $X$ such that $V_{2}$ is co-Whittaker, $\mathfrak{J}(V_{1})$ generates $V_{1}$ over $\OO_{X}[G_{S}]$, and  $V_{1}^{(n)}$ is locally free of rank $r$ (a  locally constant function on $X$). Suppose 
  that there is a  dense subset $\Sigma \subset X$ such that $$V_{1}\otimes K(x)\cong V_{2}^{\oplus r(x)}\otimes K(x)$$ for all $x\in \Sigma$. 
Then $H:= \underline{\Hom}_{\OO_{X}[G_{S}]}(V_{2}, V_{1})$ is a locally free $\OO_{X}$-module  of rank $r$, 
and the natural map
   $$ V_{2}\otimes H\stackrel{\cong}{\longrightarrow} V_{1}$$  is an isomorphism
   of $ \OO_{X}[G_{S}]$-modules.
\end{lemm}
\begin{proof}
We may reduce to the case $X=\Spec A$ is affine and connected (so that $r\in \N$) and the $V_{i}^{(n)}$ are both free over $A$.  We have $\prod_{x\in \Sigma} V_{1}\otimes K(x) \cong \prod_{x\in \Sigma} V_{2}^{\oplus r} \otimes K(x)$.  Let $\mathscr{K}=\prod_{x\in \Sigma} K(x)$;  by the argument in the proof of \cite[Lemma 6.3.7]{EH}, we deduce $V_{1}\otimes \mathscr{K}\cong V_{2}^{\oplus r}\otimes \mathscr{K}$; by torsion-freeness, both $V_{1}$ and $V_{2}^{\oplus r}$ embed in this space. We deduce that $V_{1}^{(n)}$ and $V_{2}^{(n)\oplus r} $ are both free $A$-submodules  of $V_{1}^{(n)}\otimes \mathscr{K}$ of the same rank $r$, hence $V_{1}^{(n)}\cong \gamma V_{2}^{(n) \oplus r}$ for some $\gamma\in \GL_{r}(\mathscr{K})$, and $ \mathfrak{J} (V_{1})\cong \mathfrak{J} (\gamma V_{2} ^{\oplus r}) $; taking the $A[G]$-spans of both sides,   $\gamma $ yields  $V_{1}\cong V_{2}^{\oplus r}$. As $\End_{A[G]}(V_{2})\cong A$ by Lemma \ref{schur}, we may  write $H \cong A^{\oplus r}$ and $V_{1}\cong V_{2}\otimes H$.
\end{proof}

\subsection{The universal co-Whittaker module} In this subsection, $F$ is a local field and $G=\GL_{n}(F)$; we write $\Psi:=\Psi_{v}$. 
We explain the construction of a co-Whittaker module over the Bernstein variety $\frak X_{n}$ for $G$, which will be universal in a certain sense.

\begin{prop} \lb{tre} There exists an explicit  finitely generated co-Whittaker module $\frak W$ over $\frak X_{n}$, unique up to tensoring with a line bundle with trivial $G$-action, whose fibre at any  point $x$ is given as follows.

 Let $\sg_{x}$ be the  supercuspidal support over $K(x)$ corresponding to $x$, which is a multiset of supercuspidal representations of the factors of the   Levi of a parabolic (well-defined up to conjugation) in $ G$.   Choose any ordering $(\sigma_{1}, \ldots, \sg_{r})$  on $\sg_{x}$ such that for $i<j$, $\sigma_{i}$ (a representation of $G_{d_{i}}$, identified with a length-1 segment)  does not follow $\sigma_{j}$. Let $P\subset G_{n}$ be a parabolic subgroup with Levi $G_{d_{1}}\ts\cdots \ts G_{d_{r}}$. Then
\beq\lb{desc}\frak W_{|x}={\rm I}_{P}^{G_{n}} (\sg_{1}\ot \cdots \ot \sg_{r}) =\pi_{\rm gen }((\pi_{\rm ss}^{-1}(\ot_{i}\sg_{i}), 0)).\eeq
\end{prop}

\begin{proof}[{Proof of Proposition \ref{tre}}] Uniqueness is clear from Lemma \ref{415}.

 Let 

$$\frak X_{[\frak s]}\subset \frak X_{n, \baar \Q}$$
be a connected component.
Choose (forgetting the  notation used in the statement of the proposition) a representative in $\frak s$ of the form $\sigma=\otimes\sigma_{i}^{m_{i}}$ with the $\sigma_{i}$ pairwise inertially-inequivalent representations of $G_{d_{i}}$.  If the set of unramified characters of $G_{d_{i}}$ stabilising $\sigma_{i}$ consists of those valued in $\mu_{f_{i}}$, we have 
$$\frak X_{[\frak s]} \cong \prod_{i} {\bf G}_{m}^{m_{i}} /(S_{m_{i}}\rtimes  \mu_{f_{i}}^{m_{i}}) .$$
Cover $\prod_{i}{\bf G}_{m}^{m_{i}}$ by open sets $\prod_{i} U_{w_{i}}$ for $w=(w_{i})\in  \prod_{i } S_{m_{i}}$, such that $(\chi_{1}, \ldots, \chi_{m_{i}})\in U_{i, w}$ if and only if for all $\ell_{1}<\ell_{2}$, $\chi_{i,w_{i}\ell_{1}}^{f_{i}}$ does not follow $\chi_{i, w_{i}\ell_{2}}^{f_{2}}$.   (Similarly to \eqref{sigx}, we identify a point of $x\in {\bf G}_{m}$ with the  unramified character $g\mapsto x^{v_{F}(\det g)}$, for any $G_{d}$.)
Let $P\subset G$ be a parabolic with Levi $\prod_{i}G_{d_{i}}^{m_{i}}$. Then we define a smooth  $\OO_{U_{w}}[G]$-module
\beq \lb{Ww} {\frak W}_{w}:= {\rm I}_{P}^{G} \left(\bigotimes_{i } (\sigma_{i}\chi_{ w_{i}1}^{f_{i}}\ot\cdots \ot \sigma_{i}\chi_{ w_{i}m_{i}}^{f_{i}})\right)
={\rm I}_{P'}^{G} \left( \bigotimes_{i}  \sg_{i} \ot {\rm I}_{P_{i}}^{G_{m_{i}}} (\chi_{w_{i}1}^{f_{i}} \ot\cdots  \ot  \chi_{w_{i}m_{i}}^{f_{i}})\right).
\eeq
Here $P'\supset P$ is a parabolic with Levi $\prod G_{d_{i}m_{i}}$ and $P_{i}\subset G_{m_{i}}$ is a maximal parabolic, and the second  equality  in \eqref{Ww} is the compatibility of parabolic induction in stages for $P\subset P'\subset G$.

Along $U_{ww'}$, there are unique isomorphisms $T_{w_{i}w_{i}'}\colon  {\rm I}_{P_{i}}^{G_{m_{i}}} (\chi_{w_{i}1}^{f_{i}} \ot\cdots  \ot  \chi_{w_{i}m_{i}}^{f_{i}}) \to  {\rm I}_{P_{i}}^{G_{m_{i}}} (\chi_{w'_{i}1}^{f_{i}} \ot\cdots  \ot  \chi_{w'_{i}m_{i}}^{f_{i}})$ for al l $i$,  normalised by sending $v_{0}$ to $v_{0}'$ where  $v_{0}$, $v_{0}'$ are the spherical vectors corresponding to the function in the induced representation taking the  value $1$ at $g=1$.  They induce isomorphisms
 $$T_{ww'}\colon {\frak W}_{w|U_{w}\cap U_{w'}}\to   {\frak W}_{w'|U_{w}\cap U_{w'}}$$
which define a faithfully flat descent datum of $(\frak W_{w})_{w}$  to a module over  $\X_{[\frak s]}$, and, varying $\frak s$, to a $G_{n}$-module $\frak W_{\baar\Q}$ over $\frak X_{n, \baar{\Q}}$. 

Since the functors ${\rm I}_{P}^{G}$ are Galois-equivariant, $\frak W_{\baar\Q}$ descends to an $\OO_{\X_{n}}[G]$-module $\frak W$.
It is clear from the construction that  $\frak W_{}$ is finitely generated and satisfies \eqref{desc} at all  points, therefore $\frak W^{(n)}$ is a coherent sheaf, and since it has rank $1$ at every  point by \eqref{desc}, it is locally free of rank $1$. Finally, that every quotient of $\frak W$ is non-degenerate can also be checked on the fibres, which follows from the fact they are in the image of $\pi_{\rm gen}$ and the theory recalled in \S\ref{21}. This completes the proof that $\frak W$ is co-Whittaker.

\end{proof}

\begin{defi}\label{univ coW}  The module $\frak W$  constructed by the proposition is called the  universal co-Whittaker module over $\frak X_{n}$. 
\end{defi}

We note that in \cite{helm-whitt}, another object $\frak W'$  called the universal co-Whittaker module is defined as ${\text{c-Ind}}_{N}^{G}\psi$ (over the spectrum of some field lying over $\Psi$). The results of that paper would imply that, after the appropriate  base-change, $\frak W'$ and $\frak W$ are isomorphic. This comparison is not necessary for our purposes.

The definition is partially justified by the following result.

\begin{lemm}\lb{is univ} Let $X$ be a reduced Noetherian scheme equipped with a morphism $\alpha\colon X\to {\frak X}_{n}$. Suppose given, for each generic point $\eta$ of an irreducible component of $X$, a co-Whittaker module $V_{\eta}$ over $K(\eta)$  of the form $\pi_{\rm gen}({\bf s})$ 
with supercuspidal support $\alpha(\eta)$. Then there exist:
\begin{itemize}
\item
a finitely generated,  torsion-free co-Whittaker $\OO_{X}[G_{S}]$ module $V$ over $X$, unique up to tensoring with an invertible sheaf on $X$ with trivial $G$-action, such that $V_{|\eta}=V_{\eta}$ for all $\eta$.
\item 
a surjective morphism of $\OO_{X}[G]$-modules $\alpha^{*}\frak W \to V$.
\end{itemize}
\end{lemm}
\begin{proof} Uniqueness follows from Lemma \ref{415}. 

Write $V_{\eta}= \pi_{\rm gen}((r, N)) $ where $\pi_{\rm ss}(r)$ is the supercuspidal support ${\bf s}$. By construction of $\frak W$  and Proposition \ref{ss'}, there is a surjection $\frak W_{|\eta}\ \to V_{\eta}$.

 Let $\frak W_{X}:= \alpha^{*}\frak W$ and let $V$ be the image of the  map 
$$\frak W_{X}\to \prod_{\eta} V_{\eta}.$$ 
We verify that $V$, which is clearly finitely generated,  is co-Whittaker (cf.  \cite[proof of Lemma 6.4]{helm-whitt}): every quotient of $V$ is also a quotient of the co-Whittaker module $\frak W$, hence nondegenerate; and $V^{(n)}$ is coherent, cyclic, and nonzero at all minimal primes of $X$, hence it is locally free of rank one. 
\end{proof}

\subsection{Local Langlands correspondence in families} We readopt the semiglobal notation of \S\ref{s4.1}.  We denote by $W_{F,S}:=\prod_{v\in S} W_{F_{v}}$.
  A Weil--Deligne representation of $W_{F,S}$ of dimension $n$ s a collection $r'_{S}=(r_{v}')_{v\in S}$ where $r'_{v}$ is a Weil--Deligne representation of $W_{F,v}$ of dimension $n$. We denote by $\pi_{{\rm gen},v}$ the local Langlands correspondence \eqref{pigen} for the field $F_{v}$, and extend it to a map from $n$-dimensional  Frobenius-semisimple Weil--Deligne representations of $W_{F,S}$ to smooth $G_{n,S}$-representations by
  $$\pi_{\rm gen}(r_{S}):= \ot_{v\in S} \pi_{{\rm gen}, v} (r'_{v}).$$

\begin{theo}\label{theo-LLF} Let $K$ be a field of characteristic zero and let  $r'_{S}=(r_{S}, N_{S})$ be an rank-$n$ Weil--Deligne representation of ${W}_{F,S}$ over a  reduced Noetherian  $K$-scheme  $X$.
 Then:
\begin{enumerate}
\item There exists a unique torsion-free strictly co-Whittaker
 $\OO_{X}[G_{S}]$-module  
$$\pi(r'_{S})$$
  such that for every irreducible component $X_{\eta}\subset X$ with generic point $\eta$, we have $$\pi(r'_{S})\otimes_{\OO_{X,\eta}} K(\eta)\cong  \pi_{{\rm gen}}(r'_{S|\eta}).$$ 
\item Moreover, for every $x\in X$ there is a surjection 
$$ \pi_{{\rm gen}}(r'_{S|x})\to  \pi(r'_{S})\otimes K(x),$$ unique up to multiplication by $K(x)^{\ts}$, which is an isomorphism
 if   for all $v\in S$, there exists a minimal prime $\eta$ of $\OO_{X,x}$ such that 
  $N_{v, \eta}\sim N_{v,x}$.
\end{enumerate}
\end{theo}

\begin{proof}
Uniqueness follows from   Lemma \ref{415}.  For existence, we reduce to the case where $S=\{v\}$ consists of a single prime: if $\pi(r'_{v})$ is as required by the statement with $S=\{v\}$, then  the maximal $\OO_{X}$-torsion-free
 quotient of $\otimes_{v\in S}\pi(r'_{v})$ is as required by the statement in the general case. 

We thus fix $S=\{v\}$ and omit these  subscripts from the notation.
 By Theorem \ref{coarse}, the $\OO_{X}$-representation  $r$  of $W_{F}$ defines a map $\alpha \colon X\to \frak X_{n}$. Then $\pi(r)$ is given by Lemma \ref{is univ} applied to the collection $V_{\eta}=\pi_{\rm gen}(r_{|\eta})$.

We now turn to the second part, which is still reduced to the case  where $S$ is a singleton (omitted from the notation). By \cite[Proposition 4.3.7]{EH} we may assume that $K$ is algebraically closed.  We denote by $\frakp_{y}$ the prime of $\OO_{X,x}$ corresponding to a point $y\in \Spec \OO_{X,x}$.  Let $\pi(r')_{x}$, respectively $\pi(r')_{x, \eta}$ be the pullback of $\pi(r')$ to $\OO_{X,x}$, respectively $\OO_{X,x}/'\frakp_{\eta}$. 
 Let   $$\pi^{\perp}_{\rm gen}(r'_{|y}) := \Ker( \frak W_{\OO_{X,x}}\to \pi_{\rm gen}(r'_{|y}))=  \Ker( \frak W_{\OO_{X,x/\frakp_{y}}}\to \pi_{\rm gen}(r'_{|y})). $$

 Then
\beqq
\pi(r')_{x}&\cong  \frak W_{\OO_{X,x}}/\cap_{\eta} (\frakp_{\eta}   \frak W_{\OO_{X,x}} +\pi^{\perp}_{\rm gen}(r'_{|\eta})), \\
\pi(r'_{|x}) &\cong  \frak W_{\OO_{X,x}}/ (\frakp_{x}   \frak W_{\OO_{X,x}} +\pi^{\perp}_{\rm gen}(r'_{|x}).
\eeqq
The existence of a surjection $\pi_{\rm gen}(r_{|x}')\to \pi(r')_{|x}=\pi(r'\ot \OO_{X,x})_{|x}$ is equivalent to the inclusion   $(\frakp_{x}   \frak W_{\OO_{X,x}} +\pi^{\perp}_{\rm gen}(r'_{|x})\subset (\frakp_{\eta}   \frak W_{\OO_{X,x}} +\pi^{\perp}_{\rm gen}(r'_{|\eta}))$ for all minimal primes $\eta$, that is, to the existence of a surjection $\pi_{\rm gen} (r_{|x}')\to \pi(r'\ot \OO_{X,x}/\frakp_{\eta})_{|x}$. Once established the existence of those surjections, 
 we have a factorisation 
$$\pi_{\rm gen }(r_{|x}') \to \pi(r\ot \OO_{X,x})_{|x} \to  \pi(r\ot \OO_{X,x}/\frakp_{\eta})_{|x}$$ 
for each $\eta$, so that if the composition is an isomorphism then so is the first map. In conclusion, we have reduced the assertion to be proved to the case where $\OO_{X,x}$ is a local ring. We denote by $\eta$ its unique generic point.

Suppose without loss of generality that $\alpha(\Spec \OO_{X, x})\subset \frak X_{n}$ is contained in the image of $U_{w}\subset {\bf G}_{m} ^{m_{i}}$ for $w_{i}={\rm id}$ for all $i$.  Then $\pi_{\rm gen}(r'_{\eta})=\pi({\bf s}_{\eta})$ where ${\bf s}_{\eta}$ is the multisegment 
$${\bf s }_{\eta}:= \{\bigotimes_{j=1}^{\ell_{i}} \sg_{i}\chi_{i+t_{j} |\eta}^{f_{i}} \ot \Delta(\one,t_{i, j})\}_{i}.$$
 The segment $\Delta(\one, t_{i,j})$ is a segment defined over $\Q$, and each $\sg_{i}$ is defined over $K$. Let
 $$\pi^{\circ}:= {\rm Ind}_{P}^{G} (\sg_{i} \chi_{i+t_{j}}^{f_{i}} \ot \Delta(\one,t_{i, j}))_{i},$$
a representation over $U_{w}$ which descends to $\frak X_{n, K}$ and can be pulled back to $\OO_{X,x}$; we denote by $\pi^{\circ}_{x}$ the pullback. The representation $\pi^{\circ}_{x}$ is manifestly finitely generated, torsion-free, and co-Whittaker with generic fibre $\pi_{\eta}$. By Lemma \ref{415}, $\pi^{\circ}_{x}\cong \pi(r)_{x}$. Its fibre at $x$ is $\pi((r_{|x}, N'_{|x}))=\pi_{\rm gen}({\bf s}_{x}')$, where $N_{|x}'$ is a monodromy operator on $r_{|x}$ which satisfies $N_{|x}'\sim N_{\eta}$. 
Since $N_{|x}\preceq N_{|\eta}$, by Proposition \ref{ss'} there is  a surjection
$$\pi_{\rm gen}(r_{|x}')= \pi_{\rm gen}((r_{|x}, N_{|x}))\to  \pi_{\rm gen}((r_{|x}, N_{|x}'))= \pi^{\circ}_{|x} =\pi(r)_{|x}, $$
which is unique up to scalars, and  an isomorphism if and only if $N_{|x}\sim  N_{|x}'\sim N_{\eta}$. 
\end{proof}

\begin{defi} The association
$$r'_{S}\mapsto \pi(r_{S}')$$
of Theorem \ref{theo-LLF} is called the local Langlands correspondence in families.
 When  $X=\Spec A$ is the spectrum of a $p$-adic ring  as defined above Lemma \ref{l-mon}, and $\rho_{S}$ is a continuous representation of $\prod_{v\in S}\Gal(\baar{F}_{v}/F_{v})$ over $X$, we define
$$\pi(\rho_{S}):= \pi(r_{S}'),$$
where $r_{S}'=(r_{v}')$ with $r_{v}'={\rm WD}(\rho)$.
\end{defi}

\begin{theo}\label{recognise}    Let $K$ be a field of characteristic zero and let $X$ be a reduced Noetherian scheme over $K$. Let $r_{S}'=(r_{S}, N_{S})$ be an $n$-dimensional Weil--Deligne representation of $W_{F, S}$ over $X$, and let  $ V$ be a smooth admissible finitely generated torsion-free $\OO_{X}[G_{S}]$-module   such that $$V\otimes K(x)\cong \pi_{\rm gen}(r'_{S|x})$$
 for all $x$ in a dense set $\Sigma$ of  points of $X$. Then there exist an open subset $U\subset X$ containing $\Sigma$ and an invertible sheaf $L$ over $U$ such that 
$$V|_{U}\cong L^{-1}\otimes  \pi(r'_{S})|_{U}$$
 as $\OO_{U}[G]$-modules. Explicitly, $U$ can be taken to be the maximal  open subset of $X$ such that   $V|_{U}$ is co-Whittaker, and $L
 =\underline{\Hom}(V|_{U}, \pi(r'_{S})|_{U})$.
 \end{theo}
\begin{proof} This follows from Lemma \ref{415} using Lemma \ref{coW open}.
\end{proof}

\section{$L$-functions and zeta integrals in analytic families}\label{sec: Lz}

In this section, $F$ is a non-archimedean local field,  $G_{n}:=\GL_{n}(F)$, and $X$ is a reduced Noetherian scheme over a field $K$ of characteristic zero.
\subsection{$L$-functions} Recall that if $r'=(r, N)$ is a Weil--Deligne representation over $K$, the inverse of 
$$L(0,(r,N))^{-1}:= \det(1-\phi|(\Ker N)^{I_{F}} )\in K, $$
when defined, is by definition  the value at $0$ of the $L$-function of $r'$. For $s\in \Z$, $L(s, r'):= L(0, r'(s))$. 

Recall also that there is a notion of $L$-function attached to  a smooth admissible irreducible representations $V$ of $G_{n}$ over $K$. For our purposes, we may use as definition  its compatibility with the (rational) Langlands correspondence:
$$L(s+(1-n)/2, V)^{-1}:= L^{-1}(s, \pi^{-1}(V))^{-1} \in K, \qquad s\in \Z.$$
Similarly, we may define the Rankin--Selberg  $L$-function of a pair of smooth admissible irreducible representations  representations $V_{1}$ of $G_{n_{1}}$, $V_{2}$ of $G_{n_{2}}$ by
$$L (s+(1-n_{1}n_{2})/2,V_{1}\times V_{2})^{-1} := L(s,\pi^{-1}(V)\otimes \pi^{-1}(V_{2}))^{-1} \in K, \qquad s\in \Z.$$

The following result is obvious.
\begin{prop}\label{interp Lv} Let $(r, N)$ be a  Weil--Deligne representation of $W_{F}$ over  $X$. Assume that   $N$ is  locally constant on $X$. Then 
$$L^{-1}((r,N)):= \det(1-\phi|(\Ker N)^{I_{F}} )\in     \OO(X), $$
defines a function satisfying
 $$L^{-1}((r, N))(x)= L(0, (r_{x}, N_{x}))^{-1}$$
 for all $x\in X$. 
\end{prop}
\begin{rema}\label{non-constant mon}  By Proposition \ref{nowhere dense},  any Weil--Deligne representation $r'=(r, N)$ over $X$ has locally constant monodromy  over a dense open subset of $X$ containing all pure specialisations.  

 If $(r,N)$ and $(r, N')$ are Weil--Deligne representations over $X$ such that $N\preceq_{I_{F}}N'$ (that is $N_{x}\preceq_{I_{F}}N'_{x}$ for all $x\in X$), and $X'\subset X$ is a locally closed subset over which both $N$ and $N'$ are locally constant, then $L^{-1}((r, N')|_{X'})  $ divides $L^{-1}((r, N)|_{X'})$ in $\OO(X')$. 
\end{rema} 

Suppose that  $\nu=\pi^{-1}_{I_{F}}([ \frak s])$ is an inertial type and $\frak X_{\nu}=\frak X_{[ \frak s]}\subset \frak X_{n, \baar\Q}$ is the corresponding component of the Bernstein variety, and let $\frak X_{\nu}^{\square}$ be the $ \frak X_{\nu}$-scheme carrying the universal $W_{F}$-representation $r_{\nu}$ of type $\nu$. Then $L^{-1}((r_{\nu}, 0))\in \OO(\frak X_{\nu}^{\square})$ descends to $\frak X_{\nu}$; by glueing we obtain a regular  function over all of $\frak X_{n, \baar\Q}$ which is Galois-equivariant, hence descends to a function 
 $$L_{\rm ss}^{-1}\in \OO(\frak X_{n}).$$

Similarly, we define
$$L^{-1}\in \OO(\frak X_{n}')$$
by first descending $L^{-1}((r, N)_{[\frak s],[t]}) $ via  $\frak X_{[\frak s], [t]}^{\square}\to \frak X_{[\frak s], [t]}$ (with notation as in  \eqref{univ wd}), then glueing to $\frak X_{n, \baar\Q}'$ and finally observing Galois-equivariance to descend to  $\frak X_{n}'$.

 By the compatibility of the local Langlands correspondence with $L$-functions (or by our definitions) , for each $x\in \frak X'_{n}(\C)$ corresponding to  a Well--Deligne representation $r'=(r, N)$, we have 
   $$L^{-1}_{\rm ss}(x)=L({(1-n)/ 2} ,\pi((r, 0)))^{-1}, 
   \quad L^{-1}(x)=L({(1-n)/ 2} ,\pi(r'))^{-1}, 
   $$

\subsubsection{Rankin--Selberg $L$-functions}  Let $n_{1}$, $n_{2}\geq 1$. We aim to  define  Rankin--Selberg $L$-functions $L_{\rm ss}^{-1}(\cdot , {\rm RS})$ on $\frak X_{n_{i}}\times \frak X_{{n_{2}}}$ and   $L^{-1}(\cdot , {\rm RS})$ on $\frak X'_{n_{1}}\times \frak X'_{n_{2}}$.

  First we construct  maps  ${\rm RS}\colon\frak X_{n_{1}}\times \frak X_{n_{2}}\to \frak X_{n_{1}n_{2}}$ and ${\rm RS}'\colon\frak X'_{n_{1}}\times \frak X'_{n_{2}}\to \frak X'_{n_{1}n_{2}}$.  We define the latter and the former is easier. After base-change to $\baar\Q$ we may restrict to components $\frak X_{[\frak s_{1}], [t_{1}]}\times \frak X_{[\frak s_{2}], [t_{2]}}$ corresponding to inertial types $\nu_{i}=\pi_{I}^{-1}(\frak s_{i})$  dimensions $n_{i}$ and classes of multipartitions $[t_{i}]$. Then we have a map ${\frak X}_{[\frak s_{1}], [t_{1}]}^{\square}\times\frak X_{[\frak s_{2}], [t_{2}]}^{\square}\to \frak X'_{n_{1}n_{2}}$  given, via Theorem \ref{coarseWD}, by the representation $(r,N)_{[\frak s_{1}], [t_{1}]}\otimes (r, N)_{[\frak s_{2}], [t_{2}]}$ (with the notation of \eqref{univ wd}). It is clear that this map factors through ${\frak X}_{[\frak s_{1}], [t_{1}]}\times\frak X_{[\frak s_{2}], [t_{2}]}$   and that the map obtained by glueing to all of $\frak X'_{n_{1}, \baar\Q}\times \frak X'_{n_{2}, \baar\Q}$ is Galois-equivariant and hence descends to ${\rm RS}\colon\frak X'_{n_{1}} \times \frak X'_{n_{2}}\to \frak X'_{n_{1}n_{2}}$. 
  
We can now define 
  $$L_{\rm ss}^{-1}(\cdot , {\rm RS})= {\textsc{RS}}^{*}  L_{\rm ss}^{-1},$$
    $$L^{-1}(\cdot , {\rm RS}')= {\textsc{RS}'}^{*}  L^{-1}.$$
By construction, the latter interpolates the usual $L$-functions at all points of $\frak X'_{n_{1}} \times \frak X'_{n_{2}}$.
  
For $i=1,2$ let $V_{i}=\pi(r')$ be an $\OO_{X}[G]$ module  in the essential image of the local Langlands correspondence in families of Theorem \ref{theo-LLF}.
  By  Theorem \ref{coarse}, there are maps $\alpha_{i}\colon X\to \frak X_{n_{i}}$.
    Let 
$$L_{\rm ss}^{-1}  ((1-n_{1}n_{2})/2,V_{1}\times V_{2}) :=(\alpha_{1}\times \alpha_{2})^{*}L_{\rm ss}^{-1}(\cdot, {\rm RS}).$$
Let $V_{2}({T})$ be the twist of $V_{2}$ by $\chi_{T}\circ \det$, where $\chi_{T}\colon F^{\times}\to \Q[T^{\pm1}]^{\times}$ is the unramified character $x\mapsto T^{v_{F}(x)}$.    Then we denote 
$$L_{\rm ss}^{-1}((1-n_{1}n_{2})/2,V_{1}\times V_{2}, T):=L_{\rm ss}^{-1}((1-n_{1}n_{2}/2),V_{1}\times V_{2}({T}))$$
and
$$L_{\rm ss}^{-1}(m+(1-n_{1}n_{2})/2,V_{1}\times V_{2}):=L_{\rm ss}^{-1}((1-n_{1}n_{2}/2), V_{1}\times V_{2},{q^{-m}}))$$

If moreover $V_{i}\cong \pi(r_{i}')$ for representations $r_{i}'$ with   locally constant monodromy, the maps $\alpha_{i}$ factors through maps $\alpha_{i}'\colon\X_{n_{i}}'$ and we can define $L^{-1}(\cdot , V_{1}\ts V_{2}) $ similarly to the previous paragraph.
\begin{prop}\label{RSLF} Let  $V_{i}=\pi_{i}((r_{i}, N_{i}))$ for some Weil--Deligne representations with locally constant monodromy over $X$. Then for each $m\in (1-n_{1}n_{2})/2 +\Z$ there is a unique element
$$L^{-1}  (m,V_{1}\times V_{2})\in \OO(X)$$
whose value at any $x\in X$ equals $L(m, V_{1,x}\times V_{2,x})^{-1}.$
\end{prop}
\begin{proof} Define first 
$$L^{-1}  ((1-n_{1}n_{2})/2,V_{1}\times V_{2}) := L^{-1}((r_{1}, N_{1})\otimes (r_{2}, N_{2}))= (\alpha_{1}'\times \alpha_{2}')^{*}L^{-1}(\cdot, {\rm RS}),$$
where 
$\alpha'_{i}\colon X\to \frak X_{n_{i}}'$ are the  maps given by Theorem \ref{coarseWD}. Let 
$$L^{-1}((1-n_{1}n_{2})/2,V_{1}\times V_{2}, T):=L^{-1}((1-n_{1}n_{2})/2,V_{1}\times V_{2}({T})),$$
then for any $k\in \Z$ the desired element is
$$L^{-1}(k+(1-n_{1}n_{2})/2,V_{1}\times V_{2}):=L^{-1}((1-n_{1}n_{2})/2 , V_{1}\times V_{2},{q^{-k}}).$$
\end{proof}

\subsection{Zeta integrals} We denote by $N_{n}\subset G_{n}$ the upper-triangular unipotent subgroup, by $A_{n}\subset G_{n}$ the diagonal torus, and by $K_{n}\subset G_{n}$ the maximal compact subgroup $\GL_{n}(\OO_{F})$. The Iwasawa decomposition asserts $G_{n}=N_{n}A_{n}K_{n}$.  We fix the following choices of measures. On $F$, we take the Haar measure $dx$ assigning volume $1$ to $\OO_{F}$ (then $|d_{F}|^{1/2}dx$ is the self-dual Haar measure for additive characters of level $0$, where $d_{F}\in F$ is a generator of the different ideal). On $F^{\times}$, we take the measure $d^{\times} x = \zeta_{F}(1) {dx \over |x|}$. On $G_{n}$, we take the measure $dg=\zeta_{F}(1){\prod_{i,j=1}^{n}dx_{ij}\over|\det g|^{n}}$ if $g=(x_{ij})$. On $N_{n}$, we take the measure $dn=\prod_{i>j} dx_{ij}$ if $n=(x_{ij})$. On $A_{n}$, we take the measure $da= \prod_{i=1}^{n} d^{\times}x_{ii}$ if $a=(x_{ij})$. On the quotient $N_{n}\bks G_{n}$ we take the quotient measure.

Let $X$ be a Noetherian $K$-scheme and let $V$ be an $\OO_{X}[G_{n}]$-module of Whittaker type. Let $\psi\colon F\to \OO(\Psi)^{\times}$ be the universal additive character of level zero. (We remove subscripts from the notation of \eqref{psiuniv}. The choice of the universal $\psi$ allows for specialising our treatment to any additive character of $F$ of level zero, and of course the restriction on the level could be removed.) By Proposition \ref{topder}.1,  the $\OO_{X\ts\Psi}$-module $(V^{(n)}\ot\OO(\Psi))^{\vee}$  is canonically isomorphic to $\Hom_{N_{n}}(V\ot\OO(\Psi),\psi_{}) \cong \Hom_{G_{n}}(V, {\rm Ind}(\psi))$. We denote by 
$$ \cW(V, \psi)$$ the image of the natural map 
$$(V^{(n)})^{\vee}\otimes_{\OO_{X}} V\otimes \OO(\Psi)\to {\rm Ind}_{N_{n}}^{G_{n}}\psi.$$

\subsubsection{Definition of the integrals}
For  $i=1,2$, let $V_{i}$ be an $\OO_{X}[G_{n_{i}}]$-module of Whittaker type as above, with $n_{2}\leq  n_{1}$. 
Let $W_{1}\in \W(V_{1}, \psi)$, $W_{2}\in \W(V_{2} , \psi^{-1})$, and  $m\in (1-n_{1}n_{2})/2+ \Z$, and when $n_{2}=n_{1}=n$, let $\Phi\in \mathcal{S}(F^{n}, \OO({X}))$ be a Schwartz function with values in $\OO({X})$. Denote by $G_{n}^{j}:=\{v_{F}(\det g)=j\}\subset G_{n}$. 
Define, in the respective cases $n_{2}<n_{1}$ and $0\leq k \leq n_{1}-n_{2}-1$; $n_{2}=n_{1}=n$; and $n_{2}\leq n_{1}$:
\beqq
 I^{j}_{k}(W_{1}, W_{2},m)&:=
\int_{M_{k, n_{2}}(F)}
  \int_{N_{n_{2}}\bks G_{n_{2}}^{j}}
   W_{1}\left(
   \begin{pmatrix}
   g  &&\\
   x &I_{k} & \\
   &&{I_{n_{1}-n_{2}-k}}
   \end{pmatrix}\right)
W_{2}(g) 
q_{F}^{-j(m+{-n_{1}  + n_{2}\over 2})} dg dx, \\
 I^{j}(W_{1}, W_{2},\Phi, m) &:= \int_{N_{n}\bks G_{n}^{j}} W_{1}(g) W_{2}(g)\Phi(e_{n}g)
q_{F}^{-jm} dg, \\
I_{(k)}(W_{1}, W_{2}, (\Phi), m, T)&:= \sum_{j} I_{(k)}^{j}(W_{1}, W_{2}, (\Phi), m) T^{j},
\eeqq
where we use the notation $I_{(k)}(W_{1}, W_{2}, (\Phi), m, T)$ with  parenthetical `$(\Phi), (k)$' whenever we want to treat uniformly the two cases $n_{2}<n_{1}$, $n_{2}=n_{1}$.  Our normalisation is such that when  $T=q^{-s}$ and for a suitable choice of $X$, the integral $I_{(0)}(W_{1}, W_{2}, (\Phi), m, T)$ is the zeta integral denoted by $I({s+m
}; W_{1}, W_{2}, (\Phi))$ in \cite[\S 3.2.1]{cps}.  If moreover $X=\Spec\C$, it is known that $I_{(k)}(W_{1}, W_{2}, (\Phi),m, q^{-s})$ converges in some right half-plane \cite[(2.7)]{JPPS}.

\begin{lemm} The series $I_{(k)}(W_{1}, W_{2},  (\Phi),m, T)$ belongs to $A\llb T\rrb [T^{-1}]\otimes \OO(\Psi)$.
\end{lemm}
\begin{proof}
This is  \cite[Lemma 3.1]{moss-gamma} if $n_{2}<n_{1}$.\footnote{The paper \cite{moss-gamma} assumes throughout that the local field $F$ has characteristic zero, however the proof of the cited result goes through without this assumption.}  The case $n_{2}=n_{1}$ is similar and proved as follows.
  By smoothness there exists a compact open $K'_{n}\subset K_{n}$ such that $W_{i}(gk)=W_{i}(g)$ for all $g\in G_{n}, k\in K'_{n}$. Then, by the Iwasawa decomposition, each integral $ I^{j}(W_{1}, W_{2}, \Phi, m) $ 
equals a finite sum of integrals of the from 
\beq\label{integrand}\int_{A_{n}^{j}} W_{1}'(a)W_{2}'(a)\Phi'(e_{n}a)\, da,\eeq
 where $A_{n}^{j}=
A_{n}\cap G_{n}^{j}$, and  the $W_{i}'$, $\Phi'$ are   translates of the $W_{i}$ and $\Phi$ by elements of $K_{n}$. We show that for $j\ll 0$, these integrals vanish. We use coordinates $a=\alpha(a_{1}, \ldots, a_{n}):={\rm diag}  (a_{1}\cdots a_{n},\ldots, a_{n-1} a_{n}, a_{n})$ on $A_{n}$, with $(a_{1}, \ldots, a_{n})\in (F^{\times})^{n}$. By a standard argument (see e.g. \cite[Lemma 3.2]{moss-interpol}), there is a  constant $C$, depending on the level of $W_{i}'$, such that each $W_{i}'(a)=0$ unless $v_{F}(a_{k})\geq -C$ for all $1\leq j\leq n-1$. 
As $\Phi $ is a Schwartz function, we have $\Phi(e_{n}a)=\Phi((0, \ldots, 0,a_{n}))=0 $ unless $v_{F}(a_{n})\geq  -C_{n} $ for some constant $C_{n}$. With $C'=\max\{C, C_{n}\}$, it follows that the integrand of \eqref{integrand} vanishes on $A_{n}^{j}$ whenever $j<-{n+1\choose 2}C'$. 
\end{proof}

\begin{lemm}\label{CPS4} Let $n_{2}\leq n_{1}$ and let $\frak X_{[\frak s_i], [t_{i}]}\subset \frak X_{n_{i}, \baar\Q}'$ be connected components of the extended Bernstein varieties. Let  $A_{i}:=\OO(\frak X_{[\frak s_{i}], [t_{i}]})$ and $A=A_{1}\otimes A_{2}$, let $V_{i}=\pi((r, N)_{[\frak s_{i}], [t_{i}]})$, and let $W_{1}\in \W(V_{1}, \psi)\otimes A_{2}$, $W_{2}\in A_{1}\otimes \W(V_{2} , \psi^{-1})$,  $m\in (1-n_{1}n_{2})/2+ \Z$. 

Then the Laurent series 
$$L^{-1}  (m,V_{1}\times V_{2}, T) \cdot I_{(k)}(W_{1}, W_{2}, (\Phi), m,T) $$
belongs to $ A[T]\otimes \OO(\Psi)$. Its  value at $T=1$ 
is denoted by $$L^{-1}  (m,V_{1}\times V_{2}) \cdot I_{(k)}(W_{1}, W_{2},  (\Phi) ,m).$$
\end{lemm}
\begin{proof} We may prove the result after base-change to $\C$ (which we don't signal in the notation).  Moreover, if $n_{2}<n_{1}$ we may reduce to the case $k=0$ as  in \cite[\S 2.7]{JPPS}. Recall that $\Spec A_{i}$ is the quotient by a finite group of $$\Spec \wtil{A}_{i}:={\bf G}_{m, \baar \Q}^{\sum_k{\ell_{i, k}}}$$ if $t_{i}=(t_{i,r})$ with $t_{i,r}$ a partition of length $\ell_{i, r}$. Then we may also prove the result after pullback to $\Spec \wtil{A}_{i}$. Similarly to the construction of the universal co-Whittaker module in Proposition \ref{tre},  we may cover $\Spec \wtil{A}_{i}$ with open sets $U_{j}$ such that over $U_{j}\ni (\chi_{1,1}, \ldots ,\chi_{1, \ell_{1}}, \ldots, \chi_{s,1}, \ldots, \chi_{s, \ell_{s}}) $
 we have
$$V_{i}|_{U_{i}}\cong {\rm I}_{P}^{G}(\Delta(\sigma_{1}\chi_{1,1}, t_{1,1}), \ldots, \Delta(\sigma_{1}\chi_{1, \ell_{1}}, t_{1,\ell_{1}}),\ldots,  \Delta(\sigma_{s}\chi_{s,1}, t_{s,1}), \ldots, \Delta(\sigma_{s}\chi_{s, \ell_{s}}, t_{\ell_{s}})).$$
(Namely, $U_{j}$ is the locus where the ordering of the segments occurring in the above representation satisfies the condition detailed above \eqref{rdnp}) Then the desired result  is proved by Cogdell and Piatetski--Shapiro in \cite[paragraph after Proposition 4.1]{cps}. (In their notation, $\Spec\wtil{A}_{1}$, resp.  $\Spec\wtil{A}_{2}$, is $\mathcal{D}_{\pi}$, resp. $\mathcal{D}_{\sigma}$,  where $\pi$ and $\sigma$ are given specialisations of $V_{1}$ and $V_{2}$.)
\end{proof}

\begin{prop}\label{511} Let  $X$ be a Noetherian $K$-scheme, and for  $i=1,2$, let $V_{i}$ be a co-Whittaker  $\OO_{X}[G_{n_{i}}]$-module in the essential image of the correspondence $\pi$, with $n_{2}< n_{1}$.
Let $W_{1}$, resp. $W_{2}$, be sections of $ \W(V_{1}, \psi)$,  resp. $ \W(V_{2} , \psi^{-1})$; let 
 $m\in (1-n_{1}n_{2})/2+ \Z$. 
 
Then the Laurent series 
$$L^{-1}_{\rm ss}  (m,V_{1}\times V_{2}, T) \cdot I_{(k)}(W_{1}, W_{2}, (\Phi), m,T) $$
is a section of  $ \OO_{X}[T]\otimes \OO(\Psi)$. Its  value at $T=1$ 
is denoted by $$L_{\rm ss}^{-1}  (m,V_{1}\times V_{2}) \cdot I_{(k)}(W_{1}, W_{2},  (\Phi) ,m).$$
\end{prop}
\subsubsection{Example} Suppose that $X=\Spec \C[q^{\pm s}]$ where $s$ is an indeterminate,  that $V_{i}^{\circ}=\pi((r_{i},N_{i}))$ are given irreducible generic representations of $G_{n_{i}}$ over $\C$, and that $V_{1}=V_{1}^\circ \otimes A$, $V_{2}=V_{2}^{\circ}|\det |^{s}$, the twist of $V_{2}^{\circ}$ by the unramified character $|\, |^{s}=\chi_{q^{-s}}\colon F^{\times }\to \C[q^{\pm s}]^{\times}$ sending a uniformiser to $q^{-s}$. Let $W_{i}^{\circ}$ be Whittaker functions for $V_{i}^{\circ}$ and let $W_{1}=W_{1}^{\circ}$, $W_{2}=W_{2}^{\circ}\cdot |\det|^{s}$. Let $V_{i}^{\circ, {\rm ss}}=\pi((r_{i}, 0))$. Then $I(W_{1}, W_{2}, m)$ is the zeta integral commonly denoted by $I(W_{1}^{\circ}, W_{2}^{\circ}, s+m)$ and  $L_{\rm ss}^{-1}(m,V_{1}\times V_{2})=L(s+m,V_{1}^{\circ, {\rm ss}}\times V_{2}^{\circ, {\rm ss}})^{-1}$, and with $s'=s+m$
the proposition says that 
$$ L(s', V_{1}^{\circ, {\rm ss}}\times V_{2}^{\circ, {\rm ss}})^{-1}\cdot I(W_{1}^{\circ}, W_{2}^{\circ}, s')\in \C[q^{\pm s'}].$$ 
By the theory of Rankin--Selberg zeta integrals \cite{JPPS} this statement (for all $W_{1}^{\circ}$, $W_{2}^{\circ}$) is equivalent to the  assertion that  $L(s', V_{1}^{\circ, {\rm ss}}\times V_{2}^{\circ, {\rm ss}})^{-1}$ divides  $L(s', V_{1}^{\circ}\times V_{2}^{\circ})^{-1}$ in $\C[q^{\pm s'}]$.

  \begin{proof} We may assume $X=\Spec A$ is affine and, after a possible base-change, that $K$ is algebraically closed.
   We may also assume that $\Spec A$ is connected and, after tensoring with an invertible $\OO_{X\ts\Psi}$-module, that $(V_{i}\ot\OO(\Psi))^{(n_{i})}$ is free. Let $r_{i}'$ be such that $V_{i}\cong \pi(r_{i}')$ and let  $\alpha_{i}\colon \Spec A \to \frak X_{n_{i}, K}$ be the map given by Theorem \ref{coarse}.  Suppose that $\alpha_{i}$ has image in the base-change from $\baar\Q$ to $K$ of a component $\frak X_{ [\frak s_{i}]}=\Spec A_{i}$. Let $\frak W_{i}$ be the restriction of $\frak W_{n_{i}}$ to  $\frak X_{ [\frak s_{i}]}$, then by construction of $\pi(\cdot )$ and 
Lemma \ref{is univ}
      there is a surjection $\frak W_{i}\otimes A\to V_{i}$ inducing an isomorphism $\W(\frak W_{i}, \psi)\otimes_{A_{i}, \alpha_{i}} A\cong \W (V_{i}, \psi)$. Thus we reduce to the case where $A$ is a module over $A_{1}\otimes A_{2}$ and $V_{i}=\frak W_{i}\otimes_{A_{i}}A$, and it is enough to treat the case  $A=A_{1}\otimes A_{2}$. This follows from the previous lemma.
  \end{proof}

\begin{prop}\label{515}
In the situation of  Proposition \ref{511}, suppose that  $V_{i}=\pi((r_{i}, N_{i} ))$ for a pair of families $(r_{i}, N_{i})$ of Weil--Deligne representations over $X$ with locally constant monodromy.  
Then the Laurent series 
$$L^{-1}  (m,V_{1}\times V_{2}, T) \cdot I(W_{1}, W_{2}, (\Phi), m,T) $$
is a section of  $ \OO_{X}[T]\otimes \OO(\Psi)$. Its  value at $T=1$ 
is denoted by $$L^{-1}  (m,V_{1}\times V_{2}) \cdot I(W_{1}, W_{2}, (\Phi) ,m).$$
  \end{prop}
\begin{proof} 
Similarly to the previous proof, we may reduce to  the case where $K$ is algebraically closed and $X=\Spec A$ is affine and connected. Then,   by Theorem \ref{coarseWD}, for each $i$ we have a map $\alpha_{i}'\colon X\to \frak X_{[\frak s_{i}], [t_{i}]}$  such that $V_{i}\cong \pi((r, N)_{[\frak s_{i}], [t_{i}]})  \otimes_{\OO(\frak X_{[\frak s_{i}], [t_{i}]})} A$. Thus we again reduce to Lemma \ref{CPS4}.
\end{proof}

We give an application to  local invariant inner products. 
\begin{lemm} Let $V$ be a smooth admissible irreducible and generic representation of $G_{n}$ over $\C$. 
  Suppose  that $V$ is essentially unitarisable, that is a twist of $V$ is identified with the space of smooth vectors in a unitary representation. Then there is  a $G_{n}$-invariant bilinear pairing 
$$\la\, \, \ra_{V}\colon \cW(V, \psi)\otimes \cW(V^{\vee}, \psi^{-1})\to \C$$ given by the absolutely convergent integral
$$\la W, W^{\vee}\ra_{V}= (\vol(K_{n})L(1, V\times V^{\vee}))^{-1}\cdot \int_{N_{n-1}\bks G_{n-1}} W(\smalltwomat g{}{}1) W^{\vee}(\smalltwomat g{}{}1)\, dg$$
If  $W$, $W^{\vee}$ are unramified and normalised by $W(1)=W^{\vee}(1)=1$, then $\la W, W^{\vee}\ra=1$. 
  \end{lemm}
  \begin{proof}
  We may reduce to the case where $V$ is unitarisable, then $V^{\vee}\cong \baar{V}$ and $W\mapsto \baar{W}$  defines an isomorphism $\cW(V^{\vee}, \psi^{-1})\cong \baar{\cW(V, \psi)}$. Then  $\la \ , \  \ra$ is identified with  the pairing of 
    \cite[\S 1]{euler products}. 
The second statement is proved  in Proposition 2.3
   \emph{ibid.} 
   \end{proof}

  For a scheme $X/\Q$, denote  by  $\OO_{X}(j)$  the unramified $1$-dimensional representation of $W_{F}$ sending a geometric Frobenius to $q_{F}^{-j}\in \OO(X)^{\times}$, and for a Weil--Deligne representation $r'$ denote by $r'^{*}$ its dual. When $X=\Spec \C$, by \eqref{duality} we have $\pi(r'^{*}(1-n))=\pi(r')^{\vee}$.

\begin{prop}\label{pairing} Let $r':=(r, N)$ be an $n$-dimensional Weil--Deligne representation with locally constant monodromy over $X$, let $r'^*(1-n):=\Hom_{\OO_{X}}(r', \OO_{X}(1-n))$. Let $V:= \pi(r')$, $V^{\vee}:=\pi(r'^{*}(1-n))$, so that for each $x\in X$ the representation $(V^{\vee})_{x}$ is the contragredient of $V_{x}$. 

Then there is a pairing 
$$\la\ , \, \ra\colon \cW(V, \psi)\otimes_{\OO_{X}} \cW(V^{\vee},\psi^{-1})\to \OO_{X}$$
such that for each complex geometric point $x\in X(\C)$ such that $V_{x}$ is essentially unitarisable
and any $W\in  \cW(V, \psi)$, $W^{\vee}\in \cW(V^{\vee},\psi^{-1})$, we have 
$$\la W, W^{\vee}\ra(x)= \la W_{x}, W_{x}^{\vee}\ra_{V_{x}}.$$
\end{prop}
\begin{proof} Let $\Phi\in \mathcal{S}(F^{n}, \Q)$ be a Schwartz function. By  \cite[proof of Proposition 3.1]{wei-aut} (after \cite{euler products}),  for each $x\in X(\C)$ we have
$$ \widehat{\Phi}(0)\cdot 
\int_{N_{n-1}\bks G_{n-1}} W_{x}(\smalltwomat g{}{}1) W_{x}^{\vee}(\smalltwomat g{}{}1)\, dg
 = I(W_{x}, W_{x}^{\vee}, \Phi, 1).$$
Then we may choose any $\Phi $ such that $  \widehat{\Phi}(0)=1$ and define
$$\la W, W^{\vee}\ra:= \vol(K_{n})^{-1} L^{-1}(1, V\times V^{\vee}) \cdot I(W, W^{\vee}, \Phi, 1),$$
where the right-hand side is provided by Proposition \ref{515}.
\end{proof}

\subsection{Local constants}  
In this final subsection we interpolate $\gamma$- and $\eps$-factors. 

Recall that the Deligne--Langlands $\gamma$-factor of a complex Weil--Deligne representation $r'=(r, N)$ of $W_{F'}'$, with respect to a nontrivial $\psi\colon F\to \C^{\ts}$, is
$$\gamma(r, \psi)  = \eps((r, N), \psi) {  L((r, N)^{*}(1)) \over L((r, N))}.$$
Unlike $L$- and $\eps$-factors, it is independent of the monodromy as the notation suggest; its interpolation will therefore be particularly simple.

If $r$ a representation of $W_{F}$,  we will denote by $L_{\rm ss}(r)$ the inverse (possibly equal to $\infty$) of the  function $L^{-1}((r_{?},0)) $ defined in Proposition \ref{interp Lv} . 
\begin{theo} \label{interp gamma} Let $X$ be an object of ${\rm Noeth}_{K}$ and let $r$ be a representation of $W_{F}$ on a locally free $\OO_{X}$-module of rank $n$. 
 Then there is   a meromorphic function
$$\gamma(r)\in \mathscr{K}(X\times \Psi)\cup\{\infty\}$$
such that for all complex geometric points $(x, \psi)$ of $X\times \Psi$ we have 
$$\gamma(r)(x, \psi) = \gamma(r_{x}, \psi).$$
The   divisor of $\gamma(r)$ is the product of $\Psi$ and of 
\beq
(\gamma(r))_{X}=  (L_{\rm ss}(r^{*}(1)  / L_{\rm ss}(r)).
\eeq
\end{theo}
We note that in \cite{DL}, a similar result is proved over a mixed-characteristic base in a purely Galois-theoretic way, that is  without appealing to the semisimple Langlands correspondence  in families (as we do). In fact, the result of  \emph{loc. cit.} is a key ingredient in the proof of the correspondence in families in the context of [Mixed].
\begin{proof} 
By Theorem \ref{coarse}, it suffices to construct a  function $\gamma$ on the product $\frak X_{n}\ts \Psi$ such that $\gamma(x, \psi)=\gamma(r_{x}, \psi)$ for any complex geometric point $(x, \psi)$. 

If $n=1$, the  explicit formulas for  $L(r)$ (Proposition \ref{interp Lv}) and 
$$\eps(r, \psi)=\begin{cases}  1 & \text{if $r$ is unramified}\\
\int_{{\vpi}^{-f-1}\OO^{\ts}}  r(t)^{-1}\psi(t) d_{\psi}t, &\text{if $r$ is ramified of conductor $f$}\end{cases}$$
clearly interpolate over (any connected component of) $\frak X_{1}\ts \Psi$.

Assume now $n\geq 2$. We use the argument of  Moss \cite{moss-interpol}.    Let $\mathfrak{W}'$ be the universal co-Whittaker module over $\frak X_{n}\ts \Psi$ (Definition \ref{univ coW}). We consider zeta integrals $I_{k}(W,m):= I_{k}(W, 1, m)$ for the product of  $\mathfrak{W}'$ (or its contragredient)  and the trivial representation of $G_{1}$ over $\frak X_{n}\ts \Psi$.  Define 
$$w_{n-1, 1}:= 
\begin{pmatrix}
    (-1)^n & 0 & \cdots & 0\\
0 & 0 &\cdots  & (-1)^{n-2}\\
    \vdots &\vdots  & \iddots &  \vdots\\
0 & 1 & \cdots & 0.
\end{pmatrix}
$$
and, for $W\in \mathfrak{W}'$,  ${W}^{\iota}(g):=W({}^{t}g^{-1})$, which is a Whittaker function on the contragredient. By the local functional equation of \cite{JPPS} and the compatibility of the Langlands correspondence with $\gamma$-factors,  the   function  $\gamma$   on $\frak X_{n}\ts \Psi$  is characterised, if it exists, by the property
\beq \lb{feqg} I_{n-2}(({w_{n-1,1}W)^{\iota} } , 1) = \gamma \cdot I_{0} (W, 0)\eeq
for some (equivalently for all) $W\in \mathfrak{W}$.

Let $W_{0}\in \mathfrak{W}$ be such that $I_{0}(W, 0)=1$ (e.g. we may take any $W_{0}$ whose restriction to $G_{1}\ts \{I_{n-1}\}$ equals a suitable constant multiple of the characteristic function of a compact subgroup of $F^{\ts}$). Then 
$$\gamma:= I(({w_{n-1, 1}W}_{0})^{\iota},1)$$
satisfies \eqref{feqg}, hence it
is the desired function.
\end{proof}

\begin{coro}\lb{ghj} Let $X$ be an object of ${\rm Noeth}_{K}$ and let $r'=(r, N)$ be a representation of $W_{F}'$ on a locally free $\OO_{X}$-module $M$, with locally constant monodromy. Then there is a unique section 
$$\eps(r') \in \OO(X\times \Psi)^{\times}$$
such that for all complex geometric points $(x, \psi)$ of $X\times \Psi$ we have 
$$\eps(r')(x, \psi) = \eps(r'_{x}, \psi).$$
\end{coro}
For a related result proved by a  different method, see \cite{ces}.
\begin{proof}
By Theorem \ref{coarseWD} it suffices to consider the case of the  Weil--Deligne representation $r'=(r, N)=(r, N)_{[\frak{s}], [t]}$ over  $\frak X_{[\frak s], [t]}^{\square}$ of \eqref{univ wd}, and show that the resulting function  $\eps(r')$ descends to $\frak X_{[\frak s], [t]}$.  By  Theorem  \ref{interp gamma} and Proposition \eqref{interp Lv}, we may define 
$$\eps(r') := \gamma(r)   { L(r')\over L(r'^{*}(1))} =  \gamma(r)   {L^{-1}(r'^{*}(1))  \over L^{-1}(r')}.$$
That $\eps(r')$  has neither zeros nor poles is, by Theorem \ref{interp gamma}, equivalent to the same statement for the ratio 
\beq\lb{Lratio} {L(r')  \over L(r'^{*}(1))}\cdot  {L(r^{*}(1)  \over L(r)}.\eeq
We assert that in fact $\eqref{Lratio}=\det(-\phi| r^{I_{F}}/ \Ker(N)^{I_{F}})$, which is well-defined as $\Ker(N)$ is locally free. 
 The assertion can be checked on geometric points and reduced to the case of a  representation of the form ${\rm Sp}(r, m)$, which is easy to verify.
\end{proof}

Special cases or variants of  the following corollary were  proved in   \cite{parity3},  \cite{DL}, and \cite{PX}.

\begin{coro} Let $X$ be a connected noetherian scheme over $K$ and let  $r'$ be a Weil--Deligne representation on a locally free sheaf over $X$.  Let $X^{\rm pure}\subset |X|$ be the set of those $x$ such that the monodromy filtration on  $r_{x}$  is pure. 

Suppose that there is an isomorphism $r'\cong r'^{*}(1)$.  Then the function 
\beqq
\eps \colon X^{\rm pure}&\to \{\pm 1\} \\
 x&\mapsto \eps(r'_{x}):=\eps(r'_{x}, \psi) \qquad \text{(for any $\psi\colon F\to \baar{K(x)}^{\ts}$)}
\eeqq
is constant.
\end{coro}
\begin{proof} This follows from Corollary \ref{ghj} and Proposition \ref{nowhere dense}.
\end{proof}

\end{document}